\documentclass{article}
\usepackage{amsmath, amsthm, amssymb}
\usepackage[margin=3.5cm]{geometry}
\usepackage{hyperref}
\usepackage{verbatim}
\usepackage{url}
\usepackage{yfonts}
\usepackage{youngtab}
\usepackage{shuffle}

\usepackage{tikz, xcolor, mathrsfs, multicol}
\usetikzlibrary{shapes.geometric, arrows}

\newtheorem{thm}{Theorem}[subsection]
\newtheorem{prop}[thm]{Proposition}
\newtheorem{lem}[thm]{Lemma}

\newtheorem{mydef}[thm]{Definition}
\theoremstyle{remark}
\newtheorem{rem}[thm]{Remark}
\newtheorem{notn}[thm]{Notation}

\newtheorem{exm}[thm]{Example}
\newtheorem{clm}[thm]{Claim}

\usepackage{enumerate}
\newtheorem{proc}[thm]{Procedure}
\newtheorem{conv}[thm]{Convention}

\title{Characterization of ${\cal B}(\infty)$ using marginally large tableaux and rigged configurations in the $A_n$ case via integer sequences}
\author{Roger Tian}
\date{\today}

\begin{document}
\maketitle


\begin{abstract}
Rigged configurations are combinatorial objects prominent in the study of solvable lattice models. Marginally large tableaux are semi-standard Young tableaux of special form that give a certain realization
of the crystals ${\cal B}(\infty)$. We introduce cascading sequences to characterize marginally large tableaux. Then we use cascading sequences and a known non-explicit crystal isomorphism
between marginally large tableaux and rigged configurations to give a characterization of the latter set, and to give an explicit bijection between the two sets.
\end{abstract}

\tableofcontents

\section{Introduction}
Kashiwara introduced the crystal ${\cal B}(\infty)$, which is the crystal base of the negative part $U_q^-(\mathfrak{g})$ of a quantum group, in \cite{kashi} and used it to study the Demazure crystals that were conjectured by Littelmann \cite{littel}. As ${\cal B}(\infty)$ reveals much about the structure of the quantum group $U_q(\mathfrak{g})$ itself, it is an active topic of research. By the work of Hong and Lee \cite{honglee}, ${\cal B}(\infty)$ can be realized as crystals consisting of combinatorial objects called marginally large tableaux, which are a special class of semi-standard Young tableaux.

Schilling \cite{schil} gave an explicit $U_q(\mathfrak{g})$-crystal structure to combinatorial objects called rigged configurations, which naturally serve as indexes for the eigenvalues and eigenvectors of the Hamiltonian in the Bethe Ansatz. A crystal model for ${\cal B}(\infty)$ in terms of rigged configurations was given \cite{sascrim} by Salisbury and Scrimshaw for affine simply-laced types, who also established \cite{sascrim2} an isomorphism between rigged configurations and marginally large tableaux as crystals. However, this isomorphism is not explicit, and the ${\cal B}(\infty)$ rigged configurations have not yet been explicitly characterized at the writing of \cite{hongleetian}.

The purpose of this paper is to characterize the rigged configurations of the $A_n$ type in ${\cal B}(\infty)$ and to give an explicit bijection between marginally large tableaux and ${\cal B}(\infty)$ rigged configurations of the $A_n$ type. We will achieve this by introducing special integer sequences that will be called cascading sequences. Any element of a highest weight crystal is obtained by acting on the highest weight vector via a sequence of Kashiwara operators, though this sequence is not necessarily unique. A cascading sequence can be viewed as the ``canonical'' sequence of Kashiwara operators leading to any crystal from the highest weight crystal. The desired bijection will be obtained by first establishing a bijection between the marginally large tableaux and the cascading sequences, and then establishing a bijection between the cascading sequences and the rigged configurations.

This paper is organized as follows. In Subsection \ref{mlt}, we recall the definition of marginally large tableaux. In Subsection \ref{cs}, we introduce cascading sequences and use them to characterize marginally large tableaux. In Subsection \ref{lanesdef}, we introduce an aspect of cascading sequences called lanes that will later be used in the characterization of rigged configurations. In Subsection \ref{riggedconf}, we recall the definition of ${\cal B}(\infty)$ rigged configurations in the $A_n$ case. In Subsection \ref{cs action}, we show in Lemma \ref{snake} that Kashiwara operators for rigged configurations act nicely when arranged in a cascading sequence, which allows us to obtain an interesting structural property Theorem \ref{iden rig} of rigged configurations. In Subsection \ref{riggedlanes}, we show that lanes of a cascading sequence correspond to columns of rigged partitions in the corresponding rigged configuration, and we obtain the first half of the characterization of rigged configurations Theorem \ref{constraints}. In Subsection \ref{roughidea}, we give the rough idea of our growth algorithm for characterizing rigged configurations. In Subsection \ref{plateaus}, we introduce special cascading sequences called $p$-plateaus that will be used in the growth algorithm. In Subsection \ref{addbox}, we show how to modify a cascading sequence to achieve the effect of adding boxes to a rigged partition in the corresponding rigged configuration. In Subsection \ref{growth}, we give our growth algorithm Theorem \ref{growth alg thm} for characterizing rigged configurations. In Subsection \ref{cs from rigged}, we give in Theorem \ref{get seq} an algorithm for obtaining the cascading sequence of any rigged configuration.

\section*{Acknowledgments}
The author would like to thank Anne Schilling for helpful discussions.

The author would like to thank Jin Hong, Hyeonmi Lee, and Travis Scrimshaw for reviewing parts of this paper, which began as part of a joint project with Hyeonmi Lee. However, Jin Hong and Hyenomi Lee posted a paper on the same topic to the arXiv in April, without giving due credit to RT and to strong protest from RT, even though HL had received considerable documented mathematical input from RT. This resulted in a joint paper \cite{hongleetian} (version 2) between JH, HL and RT; some form of the current writing was part of this joint paper. However, JH and HL decided unilaterally to terminate this joint project recently. Lastly, JH and HL left comments on their arXiv posting suggesting that they terminated the joint project due to serious errors and logical gaps that they uncovered in RT's paper, even though they were unable to present an example of such a serious issue when confronted on this matter by RT. This compels RT to include this description of the dispute.

The author used Sage (\cite{Sage-Combinat} and \cite{sage}) extensively to do computations with marginally large tableaux and rigged configurations that would have been forbidding by hand.

The author was partially supported by NSF grants DMS-1001256, OCI-1147247, DMS-1500050.

\section{Marginally Large Tableaux and Cascading Sequences}
\label{mlt cs}
We give a bijection between the set of $A_n$ marginally large tableaux and a special set of integer sequences that we call cascading $n$-sequences.

\subsection{Marginally Large Tableaux}
\label{mlt}
In this subsection, we recall the definition of marginally large tableaux for type $A_n$, as given in \cite{honglee}.
\begin{mydef}
We call a semi-standard tableau $T$ \textbf{large} if it has exactly $n$ nonempty row, and the $i$th row has strictly more $i$-boxes than the total number of boxes in the $(i+1)$st row, for each $1 \leq i < n$.
\end{mydef}

\begin{mydef}
By a \textbf{marginally large tableau} in the $A_n$ case we will mean a Young tableau with exactly $n$ rows whose entries come from the alphabet $J = \{1 < 2 < \cdots < n < n+1\}$ that satisfies the following conditions:
\begin{enumerate}
\item The $i$th row of the leftmost column is a single $i$-box, for each $1 \leq i \leq n$.
\item Entries increase weakly as we go from left to right along each row.
\item The number of $i$-boxes in the $i$th row exceeds by exactly one the total number of boxes in the $(i+1)$st row, for each $1 \leq i < n$. 
\end{enumerate}
\end{mydef}

Let $T(\infty)$ denote the set of $A_n$ marginally large tableaux. As shown in \cite{honglee}, $T(\infty)$ has a \textit{crystal structure}, given as follows.

\begin{proc}
We describe how to apply the Kashiwara operator $f_i$ to any marginally large tableau:
\begin{enumerate}
\item Apply $f_i$ to this tableau in the usual way, by writing the tableau as a tensor product, applying the tensor product rule, and assembling the result back into tableau form.
\item We are done if the result we obtain is a large tableau, as it will be marginally large automatically.
\item If the result we obtain is not a large tableau, then $f_i$ must have acted on the rightmost $i$-box of the $i$th row. Insert a single column of height $i$ to the left of this box that $f_i$ acted upon. For $1 \leq k \leq i$, the $k$th row of the added column must be a $k$-box.
\end{enumerate}
\end{proc}

\begin{exm}
Given the marginally large tableau \[S = \young(1111,222,34),\] we have $$f_2(S) = \young(11111,2223,34).$$
\end{exm}

\begin{proc}
We describe how to apply the Kashiwara operator $e_i$ to any marginally large tableau:
\begin{enumerate}
\item Apply $e_i$ to this tableau in the usual way.
\item We are done if the result we obtain is zero or a marginally large tableau.
\item Otherwise, the result is a large tableau that is not marginally large. $e_i$ must have acted on the box to the right of the rightmost $i$-box of the $i$th row. Remove the column that contains this changed box. This column will have height $i$, and its $k$th row consists of a single $k$-box, for $1 \leq k \leq i$.
\end{enumerate}
\end{proc}

\begin{exm}
Given the marginally large tableau \[S = \young(1111,222,34),\] we have \[e_3(S) = \young(111,22,3).\]
\end{exm}

\subsection{Cascading Sequences and a Bijection}
\label{cs}
For any $m \in [n] = \{1,2,\ldots,n\}$, we will call any subinterval $[a,m] = \{a,a+1,\ldots,m\}$ of $[n]$ an \textbf{$m$-lower subinterval}. For example, $[3,5]$ is a $5$-lower subinterval of $[6]$. By an \textbf{$m$-component}, we will mean a sequence of finitely many (allowed to be zero) $m$-lower subintervals of $[n]$ ordered by nonincreasing length.
\begin{mydef}
By a \textbf{cascading $n$-sequence} we will mean an integer sequence formed by concatenating an $n$-component, an $(n-1)$-component, an $(n-2)$-component, \ldots, in that order. Let $\bar{A}_n$ denote the set of cascading $n$-sequences.
\end{mydef}

\begin{exm}
\label{exm basic}
$(1,2,3,4,5,3,4,5,3,4,5,5,2,3,4,3,4,2,3,3,2,2,1)$ is an element of $\bar{A}_5$ where the lower subintervals (written as tuples) are $(1,2,3,4,5)$, $(3,4,5)$, $(3,4,5),$ $(5),$ $(2,3,4),$ $(3,4),$ $(2,3),$ $(3),$ $(2),$ $(2),$ $(1)$.
\end{exm}

We will follow the English notation for the Young tableau, with weakly increasing row length as we move up the tableau. Let $M^{A_n}$ denote the set of marginally large tableaux (MLT) in the $A_n$ case. We now define a map $\Phi : M^{A_n} \rightarrow \bar{A}_n$ which will be shown to be a bijection. Given a marginally large tableau $T$, we will give an $f$-string (sequence of Kashiwara operators $f_1,f_2,\ldots,f_n$; also called Lusztig data \cite{lusz}) with nice properties that gives rise to $T$ upon acting on the highest weight MLT. We will write this $f$-string as its corresponding sequence of indices, and we will see that this sequence is an element of $\bar{A}_n$. Let $T(i)$ denote the portion of the $i$th row of $T$ without the $i$ boxes. 

Define $\Phi(T)$ as follows. The $f$-string that we give will add the $(n+1)$-boxes, the $n$-boxes, the $(n-1)$-boxes, and so on in that order. Let $x_{i,j}$ denote the number of $(j+1)$-boxes in the $i$th row of $T$. The $n$-component of $\Phi(T)$ consists of $x_{i,n}$ copies of $(i,i+1,\ldots,n)$ for $i = 1,2, \ldots$. In general, the $m$-component of $\Phi(T)$ consists of $x_{i,m}$ copies of $(i,i+1,\ldots,m)$ for $i = 1,2, \ldots$; each copy of $(i,i+1,\ldots,m)$ adds an $(m+1)$-box to the $i$th row.

\begin{rem}
The cascading sequences can actually be rewritten as Berenstein-Zelevinsky-Lusztig data (in \cite{beze}, \cite{beze2}) for the reduced word $$(s_1)(s_2s_1)\ldots(s_n \ldots s_2s_1).$$ However, to the best of our knowledge, the use of such sequences for the purpose of characterization given in this paper is new.
\end{rem}

\begin{exm}
\label{test}
Given \[T = \young(11111111111111112223336,222222222222222,33333334455666,444445,5666)\] in $M^{A_5}$, we have \[\Phi(T) = (1,2,3,4,5,3,4,5,3,4,5,3,4,5,5,5,5,3,4,3,4,4,3,3,1,2,1,2,1,2,1,1,1)\] where the lower subintervals $(1,2,3,4,5),(3,4,5),(3,4,5),(3,4,5),(5),(5),(5)$ add all the $6$-boxes of $T$, the lower subintervals $(3,4),(3,4),(4)$ add all the $5$-boxes of $T$, the lower subintervals $(3),(3)$ add all the $4$-boxes of $T$, the lower subintervals $(1,2),(1,2),(1,2)$ add all the $3$-boxes of $T$, and the lower subintervals $(1),(1),(1)$ add all the $2$-boxes of $T$.
\end{exm}

\begin{prop}
\label{f-string thm}
The map $\Phi : M^{A_n} \rightarrow \bar{A}_n$ defined above is a bijection.
\end{prop}
\begin{rem}
Thus, we can take the cascading $n$-sequences to be ``canonical'' $f$-strings for $M^{A_n}$.
\end{rem}
\begin{proof}
The inverse map $\Phi^{-1}$ can be described as follows. Given an $f$-string $\alpha \in \bar{A}_n$, we can read off all its lower subintervals in left-right order. Each such lower subinterval $[i,m]$ gives an $(m+1)$-box in the $i$th row of the MLT resulting from $\alpha$ acting on the highest weight element. Thus, each such lower subinterval $[i,m]$ specifies that there must be an $(m+1)$-box in the $i$th row of $\Phi^{-1}(\alpha)$. In this way, the MLT $\Phi^{-1}(\alpha)$ is completely determined, since $\Phi^{-1}(\alpha)(i)$ is completely determined for each row $i$.
\end{proof}

\begin{rem}
Notice that the elements $\alpha$ of $\overline{A}_n$ are particularly convenient as $f$-strings for MLT's, as we can obtain the corresponding MLT $\Phi^{-1}(\alpha)$ (which is the same MLT obtained by having $\alpha$ act on the highest weight element) by simply reading off the lower subintervals of $\alpha$, \textit{without} having to apply the Kashiwara operators on the highest weight element. For instance, we see in Example \ref{test} that we can immediately obtain $T$ from the $f$-string by noting that $T$ has exactly one $6$-box in the first row specified by the lower subinterval $(1,2,3,4,5)$, exactly two $5$-boxes in the third row specified by the lower subintervals $(3,4),(3,4)$, and so on.
\end{rem}

Finally, we mention that the cascading sequence characterization in this section can also be applied to regular Young tableaux, with slight modification.

\subsection{Lanes of Cascading Sequences}
\label{lanesdef}
As already shown in \cite{sascrim}, the marginally large tableaux are isomorphic to the rigged configurations as crystals, so we can use cascading sequences to characterize the latter objects (which are in bijection with cascading sequences), which have not yet been characterized explicitly.

Given two tuples $u = (u_1,\ldots,u_i), v = (v_1,\ldots,v_j)$ we define $$u \oplus v = (u_1,\ldots,u_i,v_1,\ldots,v_j).$$ If $u, v$ are lower subintervals, we define their intersection $u \cap v$ in the natural way. For example, we have $(3,5,2) \oplus (5) = (3,5,2,5)$ and we have $(7,8,9) \cap (6,7,8) = (7,8)$.

We first introduce the aspects of cascading $n$-sequences that will be useful in describing $A_n$ rigged configurations. Let $\alpha \in \overline{A}_n$ be a cascading $n$-sequence.

For the remainder of this subsection, we partition $\alpha$ into subsequences that we will call \textbf{lanes}. As subsequences, lanes will be written as tuples. For any tuple, the first entry will be called the \textbf{head} of the tuple and the last entry will be called the \textbf{tail} of the tuple. Also, for any lane $L$ of $\alpha$, let $|L|$ denote the length of $L$. Formation of these lanes will reflect the way Kashiwara operators act on rigged configurations in Lemma \ref{snake}. Furthermore, we will show that the $i$th $l$-lane corresponds to the $i$th column of the $l$th partition in the corresponding rigged configuration. Label the lower subintervals of $\alpha$ as $I_1, I_2, \ldots, I_P$ from left to right. Denote by $I_i(j)$ the $j$th entry of $I_i$, and by $\overline{I}_i(j)$ the integer value (in $[n]$) of $I_i(j)$; in Example \ref{exm basic}, $I_5 = (2,3,4)$ and $\overline{I}_5(3) = 4$. Lanes will be formed, via the following iterative procedure, for each integer in $\alpha$; i.e. for $m \in [n]$ there will be lanes $L_1(m), L_2(m)$, \ldots at the end of the procedure. The \textbf{lane forming procedure} builds the lanes in stages, as follows:

At the outset, we form lanes using entries of $I_1$, by setting $L_1(\overline{I}_1(j)) := (I_1(j))$ for each $j$. In general, suppose a collection of lanes $M^1, M^2, \ldots, M^a$ has been formed from the lower subintervals $I_1, I_2, \ldots, I_{b-1}$. Set $L_p(q) := \emptyset$ for any $L_p(q) \notin \{M^1, M^2, \ldots, M^a\}$. We will form new lanes using entries of $I_b$. First, pick the maximal $d_1$ such that $L_{d_1}(\overline{I}_b(1)) \in \{M^1, M^2, \ldots, M^a\}$, and set $L_{d_1+1}(\overline{I}_b(1)) := (I_b(1))$; if no such $d_1$ exists, take $d_1 = 0$. In general, for any entry $I_b(k)$ with $k > 1$, pick the maximal $d_{k} \leq d_{k-1}$ such that $|L_{d_k}(\overline{I}_b(k))| > |L_{d_{k}+1}(\overline{I}_b(k))|$, and set $L_{d_k+1}(\overline{I}_b(k)) := L_{d_k+1}(\overline{I}_b(k)) \oplus (I_b(k))$; take $d_k = 0$ if no such $d_k$ exists. Finally, we fix all other preexisting lanes. At the end of this iterative procedure, we obtain the lanes partitioning $\alpha$.

\begin{exm}
\label{lanes1}
Consider the cascading $10$-sequence $$(8,9,10,8,9,10,7,8,9,7,8,9,7,8,9,8,9,6,7,8,7,8),$$ whose lower subintervals are $I_1 = (8,9,10), I_2 = (8,9,10), I_3 = (7,8,9), I_4 = (7,8,9),$ $I_5 = (7,8,9), I_6 = (8,9), I_7 = (6,7,8), I_8 = (7,8)$. The lanes are formed in the following processes (with exactly one entry added to the lane at each stage):
\begin{enumerate}
\item $L_1(8): (I_1(1)) \rightarrow (I_1(1),I_3(2)) \rightarrow (I_1(1),I_3(2),I_7(3))$
\item $L_1(9): (I_1(2)) \rightarrow (I_1(2),I_3(3))$
\item $L_1(10): (I_1(3))$
\item $L_2(8): (I_2(1)) \rightarrow (I_2(1),I_4(2))$
\item $L_2(9): (I_2(2)) \rightarrow (I_2(2),I_4(3))$
\item $L_2(10) : (I_2(3))$
\item $L_1(7): (I_3(1)) \rightarrow (I_3(1),I_7(2))$
\item $L_2(7): (I_4(1))$
\item $L_3(7): (I_5(1))$
\item $L_3(8): (I_5(2))$
\item $L_3(9): (I_5(3))$
\item $L_4(8): (I_6(1)) \rightarrow (I_6(1),I_8(2))$
\item $L_4(9): (I_6(2))$
\item $L_1(6): (I_7(1))$
\item $L_4(7): (I_8(1))$
\end{enumerate}
Written another way, the lower subintervals and lanes are $I_1 = (8^1,9^1,10^1),$ $I_2 = (8^2,9^2,10^2),$ $I_3 = (7^1,8^1,9^1),$ $I_4 = (7^2,8^2,9^2),$ $I_5 = (7^3,8^3,9^3),$ $I_6 = (8^4,9^4),$ $I_7 = (6^1,7^1,8^1),$ $I_8 = (7^4,8^4)$, where lane $i$ has been marked with a superscript $i$.
\end{exm}

\begin{exm}
\label{lanes2}
Let us now look at a more complex example. The cascading $10$-sequence 
\begin{multline*}
(6,7,8,9,10,7,8,9,10,7,8,9,10,8,9,10,6,7,8,9,6,7,8,9,7,8,9,5,6,7,8,5,6,7,8, \\ 5,6,7,8,6,7,8)
\end{multline*} 
has lower subintervals and lanes $I_1 = (6^1,7^1,8^1,9^1,10^1),$ $I_2 = (7^2,8^2,9^2,10^2),$ $I_3 = (7^3,8^3,9^3,$ $10^3),$ $I_4 = (8^4,9^4,10^4),$ $I_5 = (6^2,7^1,8^1,9^1),$ $I_6 = (6^3,7^2,8^2,9^2),$ $I_7 = (7^4,8^3,9^3),$ $I_8 = (5^1,6^1,7^1,$ $8^1),$ $I_9 = (5^2,6^2,7^2,8^2),$ $I_{10} = (5^3,6^3,7^3,8^3),$ $I_{11} = (6^4,7^4,8^4)$, where lane $i$ has been marked with a superscript $i$. \\

We now show with more detail the formation of lanes at the stage where $I_6$ is acting. \\
The lanes formed before $I_6$ are: \\
$L_1(6) = (I_1(1))$\\
$L_1(7) = (I_1(2), I_5(2))$\\
$L_1(8) = (I_1(3), I_5(3))$\\
$L_1(9) = (I_1(4), I_5(4))$\\
$L_1(10) = (I_1(5))$\\
$L_2(6) = (I_5(1))$\\
$L_2(7) = (I_2(1))$\\
$L_2(8) = (I_2(2))$\\
$L_2(9) = (I_2(3))$\\
$L_2(10) = (I_2(4))$\\
$L_3(7) = (I_3(1))$\\
$L_3(8) = (I_3(2))$\\
$L_3(9) = (I_3(3))$\\
$L_3(10) = (I_3(4))$\\
$L_4(8) = (I_4(2))$\\
$L_4(9) = (I_4(3))$\\
$L_4(10) = (I_4(4))$\\
Lastly, we have $L_3(6) = L_3(7) = L_3(8) = L_3(9) = L_4(7) = L_1(5) = L_2(5) = L_3(5) = L_4(6) = ()$.\\
Assignment of entries of $I_6$:\\
We have $L_3(6) := () \oplus (I_6(1)) = (I_6(1))$, since $d_1 = 2$ for integer value 6.\\
We have $L_2(7) := (I_2(1)) \oplus (I_6(2)) = (I_2(1), I_6(2))$, since $d_2 = 1$ for integer value 7.\\
We have $L_2(8) := (I_2(2)) \oplus (I_6(3)) = (I_2(2), I_6(3))$, since $d_3 = 1$ for integer value 8.\\
We have $L_2(9) := (I_2(3)) \oplus (I_6(4)) = (I_2(3), I_6(4))$, since $d_4 = 1$ for integer value 9.
\end{exm}

\section{Cascading Sequences and Rigged Configurations}
\label{cs riggedconf}
We use cascading sequences to give an explicit characterization (with a growth algorithm) of ${\cal B}(\infty)$ rigged configurations in the $A_n$ case, and we give an explicit bijection between these rigged configurations and cascading sequences. This results in an explicit bijection between the marginally large tableaux and $A_n$ rigged configurations.

\subsection{Rigged Configurations}
\label{riggedconf}
The definition of ${\cal B}(\infty)$ rigged configurations in the $A_n$ case is given in \cite{sascrim}, based on work done in \cite{schil}. We now recall this definition. Let $\textfrak{g}$ be a simply-laced Kac-Moody algebra with index set $I$, and let ${\cal H} := I \times \mathbb{Z}_{>0}$. Fix a multiplicity array $$L = (L_i^{(a)} \in \mathbb{Z}_{>0} : (a,i) \in {\cal H}).$$ We typically define a partition to be a multiset of integers sorted in decreasing order. Define a \textbf{rigged partition} to be a multiset of integer pairs $(i,x)$ with $i > 0$, with these pairs sorted in decreasing lexicographic order. We will call each $(i,x)$ a \textit{string}, with $i$ the \textit{size} or \textit{length} of the string and $x$ the \textit{quantum number}, \textit{label}, or \textit{rigging} of the string. By a \textbf{rigged configuration} we will mean a pair $(\nu,J)$ where $\nu = \{\nu^{(a)} : a \in I\}$ is a sequence of rigged partitions and $J = (J_i^{(a)})_{(a,i)\in {\cal H}}$ where each $J_i^{(a)}$ is the weakly increasing sequence of riggings of strings in $\nu^{(a)}$ whose length is $i$. The \textbf{vacancy number} of $\nu$ is defined as $$p_i^{(a)} = p_i^{(a)}(\nu) = -\sum_{(b,j)\in{\cal H}}{A_{ab}\min(i,j)m_j^{(b)}},$$ where $m_j^{(b)}$ is the number of parts in the partition $\nu^{(b)}$ with length $j$. The \textbf{coquantum number} or \textbf{colabel} of a string $(i,x)$ is defined to be $p_i^{(a)}-x$. The $a$th part of $(\nu,J)$ is often denoted by $(\nu,J)^{(a)}$ for brevity.

To give the definition of ${\cal B}(\infty)$ rigged configurations, denoted $RC(\infty)$, let $\nu_{\emptyset}$ be the multipartition with all parts empty; that is, set $\nu_{\emptyset} = (\nu^{(1)},\ldots,\nu^{(n)})$ where $\nu_i^{(a)} = \emptyset$ for all $(a,i) \in {\cal H}$. Therefore the rigging $J_{\emptyset}$ of $\nu_{\emptyset}$ must be $J_i^{(a)} = \emptyset$ for all $(a,i) \in {\cal H}$.
\begin{mydef}
The Kashiwara operators $e_a$ and $f_a$ act on elements $(\nu,J) \in RC(\infty)$ as follows: Fix $a \in I$,  and let $x$ denote the smallest label of $(\nu,J)^{(a)}$, assuming $(\nu,J)^{(a)} \neq \emptyset$. 
\begin{enumerate}
\item Set $e_a(\nu,J) = 0$ if $x \geq 0$. Otherwise, let $l$ denote the smallest length of all strings which have label $x$ in $(\nu,J)$. We obtain the rigged configuration $e_a(\nu,J)$ by replacing the string $(l,x)$ with $(l-1,x+1)$ and then changing all the other labels to ensure that all colabels are preserved.
\item Add the string $(1,-1)$ to $(\nu,J)^{(a)}$ if $x>0$. Otherwise, let $l$ denote the greatest length of all strings which have label $x$ in $(\nu,J)^{(a)}$. Replace the string $(l,x)$ by $(l+1,x-1)$, then change all the other labels to ensure that all colabels are preserved. The result is $f_a(\nu,J)$. \\
If $(\nu,J)^{(a)}$ is empty, then $f_a$ adds the string $(1,-1)$ to $(\nu,J)^{(a)}$. 
\end{enumerate}
$RC(\infty)$ is the graph generated by $(\nu_{\emptyset},J_{\emptyset})$ using $e_a$ and $f_a$, for $a \in I$. 
\end{mydef}

We now give the remaining part of the crystal structure: $$\epsilon_a(\nu,J) = \max\{k \in \mathbb{Z}_{\geq 0} : e_a^k(\nu,J) \neq 0\},$$ $$\phi_a(\nu,J) = \epsilon_a(\nu,J) + \langle h_{\alpha}, \mathrm{wt}(\nu,J) \rangle,$$ $$\mathrm{wt}(\nu,J) = -\sum_{(a,i) \in {\cal H}}{im_i^{(a)}\alpha_a} = -\sum_{a \in I}{|\nu^{(a)}|\alpha_a},$$ where $\{\alpha_a\}_{a \in I}$ denotes the simple roots.

\subsection{Kashiwara Operators Acting in a Cascading Sequence Arrangement}
\label{cs action}
We show in this section that $RC(\infty)$ Kashiwara operators act in a nice way when arranged in a cascading sequence. Let $R = (\nu_1,\nu_2,\ldots,\nu_n)$ be a $B(\infty)$ rigged configuration of $A_n$ type where $\nu_{i}$ is the $i$th rigged partition whose $j$th row has rigging $\mathrm{rig}_i^j$. 
\begin{notn}
Whenever we write a rigged configuration in the form $$R = (\nu_1,\nu_2,\ldots,\nu_n),$$ it is understood that each $\nu_i$ is a rigged partition carrying the riggings information $\mathrm{rig}_i^j$. \\
Fix rigged partition $\nu_m$. If $\nu_m = \emptyset$, then we regard $\nu_m$ as a single empty row $r_1$ whose length is zero and whose rigging is zero. Generally, if we label as $r_1, r_2, \ldots, r_k$ the rows of $\nu_m$ from top to bottom, then we regard $\nu_m$ has having an ``empty row'' $r_{k+1}$ beneath $r_k$, where $r_{k+1}$ is understood to have zero length and a rigging of zero.
\end{notn}
Let $\alpha$ denote the cascading sequence of $R$. Recall how the vacancy number changes when a Kashiwara operator acts on $R$: \\
If the Kashiwara operator $f_a$ adds a box to a row of length $l$ in $\nu_a$, then the vacancy numbers of $R$ are changed using the formula 
$$p_i^{(b)} =
\left\{
	\begin{array}{ll}
		p_i^{(b)}  & \mbox{if } i \leq l \\
		p_i^{(b)}-A_{ab} & \mbox{if } i > l
	\end{array}
\right.$$ 
where $p_i^{(b)}$ denotes the vacancy number of a row of length $i$ in $\nu_b$, and
$$A_{ab} =
\left\{
	\begin{array}{ll}
		-1  & \mbox{if } b = a \pm 1 \\
		2 & \mbox{if } b = a \\
		0 & \mbox{otherwise}
	\end{array}
\right.$$

For each partition $\lambda$, we will denote by $\lambda^b$ the $b$th part (row) of $\lambda$ and by $\tilde{\lambda}^b$ the portion of $\lambda^b$ that has no boxes beneath it; we call $\tilde{\lambda}^b$ the \textbf{stretch} of $\lambda^b$. For instance, $T$ in Example \ref{test} has $\tilde{T}^4 = \young(45)$. If $\lambda$ is a rigged partition, by the \textit{rigging of the stretch} $\tilde{\lambda}^b$ we will always mean the rigging of the row $\lambda^b$. Also, letting $\lambda^t$ denote the transpose of $\lambda$, $l(\lambda) := \max(\lambda^t)$ is then the number of rows $\lambda$ has.

By an integer sequence $\gamma = (\gamma_1,\gamma_2,\ldots,\gamma_p)$ \textbf{acting} on a rigged configuration $R'$ we will always mean the corresponding sequence of Kashiwara operators $\{f_i | i \in \gamma\}$ acting on $R'$. More precisely, $\gamma$ acts on $R'$ by $f_{\gamma}R' = f_{\gamma_p}f_{\gamma_{p-1}} \cdots f_{\gamma_1}R'$.

When working with cascading sequences, we can rely on the following useful lemmas:

Let $I = (a,a+1,\ldots,m)$ be an $m$-lower subinterval of the cascading sequence $\alpha$. Denote by $\alpha_I$ the subsequence of $\alpha$ before $I$. Let $R_I = (\mu_1,\mu_2,\ldots,\mu_n)$ denote the preexisting rigged configuration (corresponding to $\alpha_I$). Whenever $I$  acts on $R_I$, it adds one box to each of the partitions $\mu_a, \mu_{a+1}, \ldots, \mu_{m-1}, \mu_m$. As will be proven below, the box added to any $\mu_j$ is of two forms: contributing and noncontributing. A \textbf{contributing box} is a single box $$\framebox(10,9){}-\!1$$ which contributes $-1$ to the rigging of the row to which this box is added, contributes $-1$ to the rigging of any row of $\mu_j$ longer than the row to which it is added, and contributes $+1$ to the rigging of any row of $\mu_{j+1}$ longer than the row to which it is added, but does not change the riggings of $\mu_{b}$ for $b \neq j, j+1$. A \textbf{noncontributing box} is a single box $$\framebox(10,9){}\,0$$ with rigging $0$, which does not change the riggings of any rigged partition.

\begin{rem}
The contributing and noncontributing boxes describe the \textit{cumulative} effect of the action of $f_I$, demonstrated in Lemma \ref{snake}.
\end{rem} 

Let us analyze in more detail how $I$ acts on the preexisting rigged configuration $R_I$ corresponding to $\alpha_I$. For any partition $\lambda$ let $\overline{\lambda}$ denote the portion of $\lambda$ beneath the top row.

\begin{lem}
\label{sandwich implies}
Let $\lambda_1, \lambda_2$ be two partitions satisfying $\overline{\lambda_2} \subset \lambda_1 \subset \lambda_2$. Fix a positive integer $p$. Let $u_1$ be the uppermost row of $\lambda_1$ with length $p$ and $u_2$ be the uppermost row of $\lambda_2$ with $|u_2| \leq p$. Then every row of $\lambda_1$ below $u_1$ is no longer than $u_2$.
\end{lem}
\begin{proof}
Suppose $u_1 = \lambda_1^b$ and $u_2 = \lambda_2^c$. Since $\lambda_1 \subset \lambda_2$, we must have $c \geq b$. Since $\overline{\lambda_2} \subset \lambda_1$, we must have $c \leq b+1$. If $c = b$, then we have $|u_1| = |u_2|$, and the claim follows immediately. Suppose $c = b+1$. Then we have $|\lambda_1^d| \leq |\lambda_2^{b+1}| = |u_2|$ for any $d \geq b+1$, since $\lambda_1 \subset \lambda_2$.
\end{proof}

\begin{lem}[Main Lemma]
\label{snake}
$R_I  = (\mu_1,\mu_2,\ldots,\mu_n)$ for $I = (a,a+1,\ldots,m)$ satisfies the following properties. Let $r_a$ be the top row of $\mu_a$, and $r_i$ be the uppermost row of $\mu_i$ with $|r_i| \leq |r_{i-1}|$. 
\begin{enumerate}
\item \label{p one} The partitions $\mu_1,\mu_2,\ldots,\mu_{m-1}$ have all riggings equal to zero. For $l \in [m-1]$ we have $\overline{\mu_{l+1}} \subset \mu_l \subset \mu_{l+1}$. By Lemma \ref{sandwich implies} it follows in particular that: \\
Fix a positive integer $p$. For $l \in [m-1]$ let $u_l$ be the uppermost row of $\mu_l$ with length $p$, and let $u_{l+1}$ be the uppermost row of $\mu_{l+1}$ such that $|u_{l+1}| \leq p$. Then every row of $\mu_l$ below $u_l$ must be no longer than $u_{l+1}$.
\item \label{p two} The rows of the rigged partition $\mu_m$ above $r_m$ have the same rigging as $r_m$, and this rigging is non-positive and minimal in $\mu_m$.
\item \label{p three} $I$ acts on $R_I$ by adding a noncontributing box to the rows $r_a, r_{a+1}, \ldots, r_{m-1}$ and a contributing box to the row $r_m$.
\end{enumerate}
\end{lem}
\begin{proof}
Induction. $R_I$ clearly satisfies these properties if $I$ is the first or second lower subinterval of $\alpha$. Now consider the general case, assuming that $R_I$ satisfies these properties.

Let $R'_{I} = (\mu'_1,\mu'_2,\ldots,\mu'_n)$ denote the rigged configuration corresponding to $\alpha'_{I} := \alpha_I \oplus I$. We apply $I$ to $R_I$ to obtain $R'_{I}$ and prove that it satisfies these properties as well.

We first check Property \ref{p three} and the first statement of Property \ref{p one} for $R'_{I}$. By Property \ref{p one} for $R_I$, Kashiwara operator $f_a$ adds a box to the first row $r_a$ of $\mu_a$, adding $-1$ to its rigging, adding $+1$ to the vacancy number as well as the rigging of rows of $\mu_{a+1}$ longer than $r_a$, and not changing the riggings of $\mu_1,\ldots,\mu_{a-1}$. $f_{a+1}$ then adds a box to the uppermost row $r_{a+1}$ of $\mu_{a+1}$ with $|r_{a+1}| \leq |r_a|$, adding $-1$ to its rigging, adding $-2$ to the vacancy number as well as the rigging of rows of $\mu_{a+1}$ longer than $r_{a+1}$ (so these rows end up with a rigging of $1 - 2 = -1$), adding $+1$ to the vacancy number and the rigging of rows of the $a$th partition longer than $r_{a+1}$ (which by Property \ref{p one} gives $\mu'_a$ with zero riggings), and adding $+1$ to the vacancy number as well as the rigging of rows of $\mu_{a+2}$ longer than $r_{a+1}$. $f_{a+2}$ then adds a box to the uppermost row $r_{a+2}$ of $\mu_{a+2}$ with $|r_{a+2}| \leq |r_{a+1}|$, adding $-1$ to its rigging, adding $-2$ to the vacancy number and the rigging of rows of $\mu_{a+2}$ longer than $r_{a+2}$ (so these rows end up with a rigging of $1 - 2 = -1$), adding $+1$ to the vacancy number and the rigging of rows of the $(a+1)$st partition longer than $r_{a+2}$ (which by Property \ref{p one} gives $\mu'_{a+1}$ with zero riggings), and adding $+1$ to the vacancy number as well as the rigging of rows of $\mu_{a+3}$ longer than $r_{a+2}$. Iterating this process, for $j = 0, 1, \ldots, m-a-2$ we obtain $\mu'_{a+j}$ by adding a noncontributing box to row $r_{a+j}$ of $\mu_{a+j}$ so $\mu'_{a+j}$ has zero riggings. 

Now, after $f_{m-1}$ added a box to row $r_{m-1}$ of $\mu_{m-1}$, all rows of the resulting $(m-1)$st partition with length at least $|r_{m-1}|+1$ have rigging $-1$. By Property \ref{p two} for $R_I$, this action of $f_{m-1}$ must have contributed $+1$ to the vacancy number and the rigging of all rows of $\mu_m$ longer than $r_{m-1}$, and consequently these rows of $\mu_m$ now have greater rigging than $r_m$ does, so $r_m$ is now the longest row of $\mu_m$ with the smallest rigging. Finally, $f_m$ adds a box to $r_m$, adding $-1$ to its rigging, adding $+1$ to the vacancy number and the rigging of rows of the $(m-1)$st partition longer than $|r_m|$ (which by Property \ref{p one} gives $\mu'_{m-1}$ with zero riggings), adding $-2$ to the vacancy number and the rigging of rows of the $m$th partition longer than $r_m$ (so these rows now have the same rigging as rows of length $|r_m|+1$), and adding $+1$ to the vacancy number as well as the rigging of rows of $\mu_{m+1}$ longer than $|r_m|$. This shows that $\mu'_{m-1}$ is obtained from $\mu_{m-1}$ by adding a noncontributing box to $r_{m-1}$, and that $\mu'_m$ is obtained from $\mu_m$ by adding a contributing box to $r_m$.

Now we verify Property \ref{p two} for $R'_I = (\mu'_1,\mu'_2,\ldots,\mu'_n)$. By above, we conclude that rows of $\mu'_m$ with length at least $|r_m| + 1$ have identical rigging, and this rigging is minimal and non-positive. Let $a' \geq a$. Let $r'_{a'}$ be the top row of $\mu'_{a'}$, and let $r'_{k}$ be the uppermost row of $\mu'_{k}$ with $|r'_k| \leq |r'_{k-1}|$, for $k = a'+1, a'+2, \ldots, m$. Notice that we have $|r'_{a'}| = |r_a|+1$ if $a' = a$ and we have $|r'_{a'}| \geq |r_a|+1$ if $a' > a$. Since $\mu'_{k}$ contains a row with length $|r_k|+1$, we have $|r'_k| \geq |r_k|+1$, for $k = a', a'+1, \ldots, m$. It follows that the rows of $\mu'_m$ above $r'_m$ have the same rigging as $r'_m$, and this rigging is non-positive and minimal in $\mu'_m$.

Finally, we verify the second statement of Property \ref{p one} for $R'_I = (\mu'_1,\mu'_2,\ldots,\mu'_n)$. If $l, l+1 < a$, then the claim follows by hypothesis. If $l = a-1$, then the claim follows since $\mu'_{l+1} = \mu'_a$ is obtained from $\mu_a$ by adding a single box to the first row. Now suppose $l \in [a,m-1]$. $\mu'_l, \mu'_{l+1}$ are obtained from $\mu_l, \mu_{l+1}$ respectively by adding a box via Property \ref{p three}. Let $r_l = \mu_l^c, r_{l+1} = \mu_{l+1}^d$ be the rows of $\mu_l, \mu_{l+1}$ respectively to which the box was added. Then $r_{l+1}$ is the uppermost row of $\mu_{l+1}$ no longer than $r_l$. Since $\mu_l \subset \mu_{l+1}$, we must have $d \geq c$. Since $\overline{\mu_{l+1}} \subset \mu_l$, we must have $d \leq c+1$. Thus, either $d = c$ or $d = c+1$. Suppose $d = c$. Then $|r_l| = |r_{l+1}|$, and it follows immediately that $\mu'_l \subset \mu'_{l+1}$. If $r_{l+1}$ is the first row, then $\overline{\mu'_{l+1}} \subset \mu'_l$ by hypothesis. If $r_{l+1}$ is not the first row, we still have $\overline{\mu'_{l+1}} \subset \mu'_l$ since $|(\mu'_l)^{c-1}| = |\mu_l^{c-1}| \geq |r_{l+1}|+1 = |(\overline{\mu'_{l+1}})^{c-1}|$. Suppose $d = c+1$. Then we must have $|r_l| < |\mu_{l+1}^c|$, so $|(\mu'_l)^c| = |r_l| + 1 \leq |\mu_{l+1}^c| = |(\mu'_{l+1})^c|$ and thus $\mu'_l \subset \mu'_{l+1}$. Since $|r_{l+1}| = |\mu_{l+1}^{c+1}| \leq |\mu_l^c| = |r_l|$, we have $|(\mu'_{l+1})^{c+1}| = |r_{l+1}| + 1 \leq |\mu_l^c| + 1 = |(\mu'_l)^c|$, and thus we have $\overline{\mu'_{l+1}} \subset \mu'_l$. This completes the induction. 
\end{proof}
\begin{rem}
\label{bar rem}
It is easy to see that the first containment $\overline{\mu_{l+1}} \subset \mu_l$ in fact holds for all $l \in [n-1]$, since no more boxes will be added to the $(l+1)$st partition once all the $(l+1)$-lower subintervals have acted.
\end{rem}

It follows immediately from Lemma \ref{snake} that
\begin{lem}
\label{unif con} The following are true.
\begin{enumerate}
\item If $\alpha$ ends in a $p$-lower subinterval, then $\nu_q$ of $R$ has zero riggings for all $q \leq p-1$.
\item All contributing boxes (and hence negative riggings) to the $\nu_m$ of $R$ are added by $m$-lower subintervals of $\alpha$.
\item All positive contributions to the riggings of $\nu_m$ are added by $(m-1)$-lower subintervals of $\alpha$, which add no boxes to $\nu_m$.
\item \label{single strip}Suppose $I_1, I_2$ are $m$-lower subintervals of $\alpha$ with $I_1$ preceding $I_2$. If $I_1$ adds a contributing box to column $i_1$ and $I_2$ adds a contributing box to column $i_2$ of the $m$th partition, then we have $i_1 < i_2$.
\end{enumerate}
\end{lem}
\begin{proof}
The first three items follow immediately from the lemma. For the fourth item, note that $I_1 = (a_1,\ldots,m)$ and $I_2 = (a_2,\ldots,m)$, where $a_1 \leq a_2$. Let $r_{a_1}$ be the top row of $\mu_{a_1}$, and $r_i$ be the uppermost row of $\mu_i$ with $|r_i| \leq |r_{i-1}|$. Let $r'_{a_2}$ be the top row of $\mu_{a_2}$, and $r'_i$ be the uppermost row of $\mu_i$ with $|r'_i| \leq |r'_{i-1}|$. Notice that $|r'_k| > |r_k|$ for $k = a_2, a_2+1, \ldots, m$. Applying Property \ref{p three} of Lemma \ref{snake}, we deduce that the contributing box added by $I_2$ must be strictly to the right of the contributing box added by $I_1$.
\end{proof}

We thus obtain the following interesting result.
\begin{thm}
\label{iden rig}
Identical rows of $\nu_m$ of $R$ have equal riggings, for any $m \in [n]$. 
\end{thm}
\begin{proof}
Before any $m$-lower subinterval has acted, the $m$th partition has zero riggings. After the first $m$-lower subinterval adds a contributing box to row $r$, every row with length at least $|r|+1$ has rigging $-1$, while the rigging of every row with length at most $|r|$ remains unchanged. In general, assume that the $j$th $m$-lower subinterval has added a box to row $r'$ of the $m$th partition, so that rows with length at least $|r'|+1$ have equal rigging, and that identical rows with length at most $|r'|$ have equal rigging. By the fourth item of Lemma \ref{unif con}, the $(j+1)$st $m$-lower subinterval adds a contributing box to row $r''$ with $|r''| \geq |r'|+1$. In the resulting $m$th partition, rows of length at most $|r''|$ have unchanged riggings, which is the same for identical rows, while the new row with length $|r''|+1$ and other rows with length at least $|r''|+1$ receive a $-1$ contribution to their identical riggings. This shows that the riggings of identical rows remain equal after all the $m$-lower subintervals have acted.

Similarly, each time an $(m-1)$-lower subinterval acts, all the rows of $\nu_m$ no longer than a certain length $k$ experience no change in rigging, while all the rows of $\nu_m$ longer than $k$ receive $+1$ contribution to the rigging. Therefore, the riggings of identical rows remain equal after all the $(m-1)$-lower subintervals have acted.
\end{proof}

\subsection{Obtaining the Rigged Configuration from the Cascading Sequence}
\label{riggedlanes}
Now we relate the concepts in Subsection \ref{lanesdef} to the ${\cal B}(\infty)$ rigged configurations in the $A_n$ case. Let $R = (\nu_1,\nu_2,\ldots,\nu_n)$ be an $A_n$ rigged configuration. Let $\alpha$ denote the cascading sequence of $R$. We can obtain any partition in the corresponding rigged configuration without doing explicit calculation via the Kashiwara operators involved. This is done by partitioning $\alpha$ into lanes and then analyzing the relevant lanes.

Each column of $\nu_l$ ends in exactly one of the stretches of $\nu_l$. We denote by $\mathrm{col}(\tilde{\nu_l}^b)$ the set of columns of $\nu_l$ ending in the stretch $\tilde{\nu_l}^b$, and let $W_b^l := \{$lanes $L_i(l) | |L_i(l)| = b\}$.

By Lemma \ref{snake}, the $l$-lanes correspond precisely to the columns of $\nu_l$, and we have the following useful facts:

\begin{lem}
\label{lane col}
The set $\mathrm{col}(\tilde{\nu_l}^b)$ corresponds to the set $W_b^l$. In fact, under the definition of lanes, $L_i(l)$ corresponds exactly to the $i$th column of $\nu_l$, with the height of the column given by $|L_i(l)|$. 
\end{lem}
\begin{proof}
Adding a box to the longest row $r$ with $|r| \leq c$ is the same as adding a box to the $d$th column for maximal $d \leq c+1$ whose height is strictly less than that of the $(d-1)$st column.
\end{proof}
\begin{rem}
In particular, the number of columns of height $b$ in $\nu_l$ is given by the number of $l$-lanes $L_i(l)$ with $|L_i(l)| = b$ in the corresponding cascading $n$-sequence.
\end{rem}

Roughly speaking, the riggings of $\nu_l$ are determined by the number of $l$-lanes that contain the right endpoint of some lower subinterval and by the number of $(l-1)$-lanes that contain the right endpoint of some lower subinterval. In Example \ref{lanes2}, if we fix $l = 9$, then the $9$-lane $L_2(9)$ is an example of the former because it contains the right endpoint of $I_5$, and the $8$-lane $L_3(8)$ is an example of the latter because it contains the right endpoint of $I_9$.

\begin{lem}
\label{lane rig}
Suppose $r$ is a row of $\nu_l$. Let $V^l_r$ denote the set of $l$-lanes $L_i(l)$ ending at a right endpoint, where $i \leq |r|$. Let $V^{l-1}_r$ denote the set of $(l-1)$-lanes $L_{i}(l-1)$ ending at a right endpoint, where $i \leq |r|$. Then the rigging of $r$ is given by $-|V^l_r| + |V^{l-1}_r|$.
\end{lem}
\begin{proof}
Lemma \ref{unif con} gives us that at most one contributing box can be added to a column. The rigging of $r$ is determined by the number of contributing boxes (negative contribution) that were added to the columns of the $l$th partition occupied by $r$ and the number of contributing boxes (positive contribution) that will have been added to the same columns of the $(l-1)$st partition; the former number corresponds to the term $-|V^l_r|$, while the latter number corresponds to the term $|V^{l-1}_r|$.
\end{proof}

For any entry $\chi$ of $\alpha$ with value $|\chi|$, if $\chi$ is the $v$th entry of the lane $L_u(|\chi|)$, we say that $\chi$ has \textbf{lane depth} $v$, $\chi$ has \textbf{lane number} $u$, and we also refer to $u$ as the \textbf{lane number} of $L_u(|\chi|)$. In Example \ref{lanes2}, the entry $8$ of $I_9$ has lane number $2$ and depth $3$.

\begin{lem}
\label{max rows}
$\nu_l$ has at most $\mathrm{maxr}_l := \min(n-l,l-1)+1 = \min(n-l+1,l)$ rows.
\end{lem}
\begin{proof}
It suffices to consider $L_1(l)$, which corresponds to the first column of $\nu_l$. Each entry of $L_1(l)$ is an entry of some lower subinterval $I$ with $\min I \in [l]$ and $\max I \in [l,n]$. Suppose that $L_1(l)(j), L_1(l)(j+1)$ are contained in an $m_1$-lower subinterval, an $m_2$-lower subinterval respectively. By Lemma \ref{unif con} Property \ref{single strip}, we must have $m_1 \neq m_2$. Since the $m_1$-lower subintervals must precede the $m_2$-lower subintervals, it follows that $m_2 < m_1$. Hence we have $|L_1(l)| \leq |[l,n]| = n-l+1$. On the other hand, notice that $\nu_1$ can only be a single row, by Lemma \ref{snake}. By Lemma \ref{snake}, for any lower subinterval $I'$, the lane depth of entry $b+1$ of $I'$ exceeds the lane depth of entry $b$ of $I'$ by at most one. 
Thus, inductively we have $|L_1(l)| \leq |[l]| = l$ as well.
\end{proof}

\begin{lem}
\label{max cases}
If $l > \frac{n+1}{2}$, then $\mathrm{maxr}_l = n-l+1$. If $l \leq \frac{n+1}{2}$, then $\mathrm{maxr}_l = l$.
\end{lem}
\begin{proof}
If $l > \frac{n+1}{2}$, then $n-l+1 < n-\frac{n+1}{2}+1 = \frac{n+1}{2} < l$. If $l \leq \frac{n+1}{2}$, then $n-l+1 \geq n - \frac{n+1}{2}+1 = \frac{n+1}{2} \geq l$. The claims then follow by Lemma \ref{max rows}.
\end{proof}

\begin{lem}
\label{nondec max row}
We have $\mathrm{maxr}_l \leq \mathrm{maxr}_{l-1}$ if and only if $l-1 > n-l$ (equivalently $l > \frac{n+1}{2}$).
\end{lem}

\begin{exm}
Consider the element $$(7,8,9,10,7,8,9,10,8,9,10,6,7,8,9,6,7,8,9,7,8,9,5,6,7,8,5,6,7,8,7,8)$$ of $\overline{A}_{10}$, whose lower subintervals are $(7^1,8^1,9^1,10^1),$ $(7^2,8^2,9^2,10^2),$ $(8^3,9^3,10^3),$ $(6^1,7^1,8^1,9^1),$ $(6^2,7^2,8^2,9^2),$ $(7^3,8^3,9^3),$ $(5^1,6^1,7^1,8^1),$ $(5^2,6^2,7^2,8^2),$ $(7^4,8^4)$, where the lanes have been marked with superscripts. From this information, we can tell, for example, that the $9$th partition of this rigged configuration has exactly three columns of height $2$, the $8$th partition has exactly one column of height $2$, and the $7$th partition has height $1$ for both its third and fourth columns. 
\end{exm}

We can apply Lemmas \ref{lane col} and \ref{lane rig} to obtain the $l$th partition in the rigged configuration as well as its riggings, given the corresponding cascading $n$-sequence. We illustrate this in the following
\begin{exm}
In the corresponding rigged configuration in Example \ref{lanes1}, the $10$th partition is $$\small\young(\ \ )$$ with rigging $-2+2 = 0$, since $L_1(10), L_2(10)$ end at right endpoints (contributing $-1-1$ to the rigging) and since $L_1(9), L_2(9)$ also end at right endpoints (contributing $+1+1$ to the rigging). \\
The $9$th partition is $$\small\young(\ \ \ \ ,\ \ )$$ with rigging $-1$ for the second row and rigging $-2$ for the first row, since $L_1(9),$ $L_2(9),$ $L_3(9),$ $L_4(9)$ all end at right endpoints (contributing $-1-1$ to the rigging of the second row and $-1-1-1-1$ to the rigging of the first row) with $|L_1(9)| = |L_2(9)| = 2$ and $L_3(9) = L_4(9) = 1$, and since $L_1(8), L_4(8)$ end at right endpoints (contributing $+1$ to the rigging of the second row and $+1+1$ to the rigging of the first row). \\
Similarly, the $8$th partition is $$\small\young(\ \ \ \ ,\ \ \ ,\ )$$ with riggings $-1$ for the third row, $-2$ for the second row, and $-2$ for the first row.
\end{exm}
\begin{exm}
In Example \ref{lanes2}, the cascading $10$-sequence 
\begin{multline*}
(6,7,8,9,10,7,8,9,10,7,8,9,10,8,9,10,6,7,8,9,6,7,8,9,7,8,9,5,6,7,8,5,6,7,8, \\ 5,6,7,8,6,7,8)
\end{multline*}  
has lower subintervals and lanes $I_1 = (6^1,7^1,8^1,9^1,10^1),$ $I_2 = (7^2,8^2,9^2,10^2),$ $I_3 = (7^3,8^3,9^3,$ $10^3),$ $I_4 = (8^4,9^4,10^4),$ $I_5 = (6^2,7^1,8^1,9^1),$ $I_6 = (6^3,7^2,8^2,9^2),$ $I_7 = (7^4,8^3,9^3),$ $I_8 = (5^1,6^1,7^1,$ $8^1),$ $I_9 = (5^2,6^2,7^2,8^2),$ $I_{10} = (5^3,6^3,7^3,8^3),$ $I_{11} = (6^4,7^4,8^4)$, where lane $i$ has been marked with a superscript $i$. \\
Looking at these lanes, we can tell that
\begin{enumerate}
\item $\nu_{10}$ has four columns of length one, with $\mathrm{rig}_{10}^1 = -4 + 3 = -1$, since $L_1(10), \ldots,$ $L_4(10)$ end at right endpoints, and since $L_1(9), L_2(9), L_3(9)$ end at right endpoints
\item $\nu_9$ has three columns of length two, and one column of length one, with $\mathrm{rig}_9^2 = -3 + 3 = 0$ and $\mathrm{rig}_9^1 = -3 + 3 + 1 = 1$, since $L_1(9), L_2(9), L_3(9)$ end at right endpoints but $L_4(9)$ does not, and since $L_1(8), \ldots, L_4(8)$ end at right endpoints
\item $\nu_8$ has three columns of length three and one column of length two, with $\mathrm{rig}_8^3 = -3$ and $\mathrm{rig}_8^2 = \mathrm{rig}_8^1 = -3 - 1 = -4$, since $L_1(8), \ldots, L_4(8)$ end at right endpoints
\end{enumerate}
The rigged configuration (with $\nu_i$ in top-bottom order) in its entirety is \\
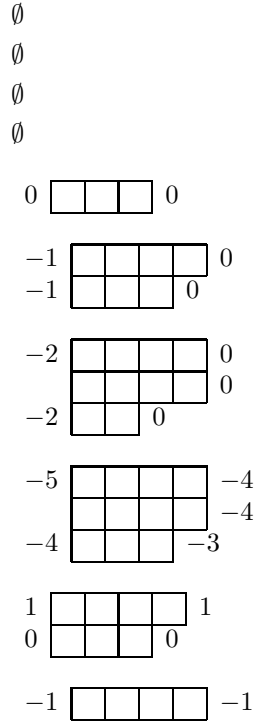
\begin{align*}
&\emptyset\\
\vspace{5 mm}
&\emptyset\\
\vspace{5 mm}
&\emptyset\\
\vspace{5 mm}
&\emptyset\\
\vspace{5 mm}
 &{
\begin{array}[t]{r|c|c|c|l}
 \cline{2-4} 0 &\phantom{|}&\phantom{|}&\phantom{|}& 0 \\
 \cline{2-4} 
\end{array}
} \\
\vspace{5 mm}
 &{
\begin{array}[t]{r|c|c|c|c|l}
 \cline{2-5} -1 &\phantom{|}&\phantom{|}&\phantom{|}&\phantom{|}& 0 \\
 \cline{2-5} -1 &\phantom{|}&\phantom{|}&\phantom{|}& \multicolumn{2 }{l}{ 0 } \\
 \cline{2-4} 
\end{array}
} \\
\vspace{5 mm}
 &{
\begin{array}[t]{r|c|c|c|c|l}
 \cline{2-5} -2 &\phantom{|}&\phantom{|}&\phantom{|}&\phantom{|}& 0 \\
 \cline{2-5}  &\phantom{|}&\phantom{|}&\phantom{|}&\phantom{|}& 0 \\
 \cline{2-5} -2 &\phantom{|}&\phantom{|}& \multicolumn{3 }{l}{ 0 } \\
 \cline{2-3} 
\end{array}
} \\
\vspace{5 mm}
 &{
\begin{array}[t]{r|c|c|c|c|l}
 \cline{2-5} -5 &\phantom{|}&\phantom{|}&\phantom{|}&\phantom{|}& -4 \\
 \cline{2-5}  &\phantom{|}&\phantom{|}&\phantom{|}&\phantom{|}& -4 \\
 \cline{2-5} -4 &\phantom{|}&\phantom{|}&\phantom{|}& \multicolumn{2 }{l}{ -3 } \\
 \cline{2-4} 
\end{array}
} \\
\vspace{5 mm}
 &{
\begin{array}[t]{r|c|c|c|c|l}
 \cline{2-5} 1 &\phantom{|}&\phantom{|}&\phantom{|}&\phantom{|}& 1 \\
 \cline{2-5} 0 &\phantom{|}&\phantom{|}&\phantom{|}& \multicolumn{2 }{l}{ 0 } \\
 \cline{2-4} 
\end{array}
} \\
\vspace{5 mm}
 &{
\begin{array}[t]{r|c|c|c|c|l}
 \cline{2-5} -1 &\phantom{|}&\phantom{|}&\phantom{|}&\phantom{|}& -1 \\
 \cline{2-5} 
\end{array}
}
\end{align*}
\end{exm}

Lastly, the following theorem imposing constraints (in a recursive manner, starting from the last partition) on the range of possible legitimate ${\cal B}(\infty)$ rigged configurations of type $A$ also follows from Lemma \ref{snake} (which states that at most one noncontributing box can be added to each column), Lemma \ref{unif con} Property \ref{single strip} (which states that at most one contributing box can be added to each column), and Lemma \ref{max rows}. This result is the first half of our classification of rigged configurations. For convenience of description, we will regard $\overline{\nu_m}$ as having $|\nu_m^1| - |\nu_m^2|$ columns of height zero to the right of its first row (i.e. $\overline{\nu_m}$ has an empty row of length $|\nu_m^1| - |\nu_m^2|$ to the right of its first row).
\begin{thm}
\label{urconstraints}
$\nu_{m-1}$ is obtained from $\nu_m$ in three stages:
\begin{enumerate}
\item Add at most one noncontributing box to each column of $\overline{\nu_m}$, resulting in a partition $\widehat{\nu_{m-1}}$.
\item Add at most one contributing box to each column of $\widehat{\nu_{m-1}}$, resulting in a partition $\widehat{\nu_{m-1}}'$.
\item Finally add a number of contributing boxes to the first row of $\widehat{\nu_{m-1}}'$.
\end{enumerate}
In this process, any column of $\overline{\nu_m}$ with height $\max((\overline{\nu_m})^t)$ must receive at most $\min(\mathrm{maxr}_{m-1} - \max((\overline{\nu_m})^t),2)$ boxes.
\end{thm}

We will use the following restatement extensively.

\begin{thm}[Restatement of Theorem \ref{urconstraints}]
\label{constraints}
$\nu_{m-1}$ is obtained from $\nu_m$ by adding at most two boxes to each column of $\overline{\nu_m}$, the first box being noncontributing and the second box being contributing. Moreover, we have the following constraints:
\begin{enumerate}
\item \label{`} At most $\min(\mathrm{maxr}_{m-1} - \max((\overline{\nu_m})^t),2)$ boxes can be added to any column of height $\max((\overline{\nu_m})^t)$.
\item At most one box can be added to the $d$th column for $d > |\nu_m^1|$, and this box must be contributing.
\item In any row of $\nu_{m-1}$, no contributing box precedes a noncontributing box.
\end{enumerate}
\end{thm}
\begin{rem}
Item \ref{`} simply states that the resulting $(m-1)$st partition cannot have more rows than $\mathrm{maxr}_{m-1}$. Any box added to a column of $\overline{\nu_m}$ with height zero will be in the first row of the resulting partition. Needless to say, there cannot be any gaps between the boxes added to any row, as the result would not be a valid partition.

Even though we have not specified the rigging of $\nu_m$ here, this theorem gives us the ``at most two boxes to each column'' constraint. Precisely how the rigging of $\nu_m$ constrains $\nu_{m-1}$ will be handled in later sections. 
\end{rem}

\subsection{Rough Idea of the Algorithm}
\label{roughidea}
Given an $m$-lower subinterval $I = (a,a+1,\ldots,m)$, we say that the lower subinterval $I_+ = (a-1,a,a+1,\ldots,m)$ is the \textbf{lengthening} of $I$, and we say that we \textbf{lengthen} $I$ to obtain $I_+$.

Our characterization for the $A_n$ rigged configurations will be an algorithm for growing rigged configurations starting from the last ($n$th) rigged partition; this growth algorithm can determine whether any given $n$-tuple of rigged partitions is a legitimate $A_n$ rigged configuration. In other words, given the last partition (which consists of a row with any number of boxes), we can give the range of all possible $(n-1)$st partitions and its riggings. In general, given the $n$th, $(n-1)$st, ..., $(n-i)$th partitions, we can give the range of all possible $(n-i-1)$st partitions and its riggings.

Growing the rigged configuration in our algorithm corresponds to growing its corresponding cascading $n$-sequence. Note that any cascading $n$-sequence can be constructed by first adding copies of $n$ to the (initially empty) string, then copies of $n-1$ to the string, then copies of $n-2$ to the string, and so on, such that we have a cascading $n$-sequence at each stage. It follows \textit{a fortiori} that any cascading $n$-sequence can be constructed by first adding copies of $i \leq n$ to the (initially empty) string, then copies of $i \leq n-1$ to the string, then copies of $i \leq n-2$ to the string, and so on, such that we have a cascading $n$-sequence at each stage. Hence any $A_n$ rigged configuration can be constructed (by applying the Kashiwara operators in the order of the cascading $n$-sequence at each stage) via this type of iterative process, which constructs the $n$th partition, $(n-1)$st partition, $(n-2)$nd partition, and so on, in that order. What we need to do is to fine tune this process so that the already constructed $n$th, $(n-1)$st, ..., $(n-i)$th partitions and their riggings do not change when we construct the $(n-i-1)$st partition. More precisely, at the $i$th stage, we will add all the copies of $n-i+1$ along with minimal copies of $j < n-i+1$ necessary to preserve the previously constructed rigged partitions; we will elaborate on this in the next few subsections.

\begin{rem}
We mention that, by Lemma \ref{lane col}, if $\nu$ is the $l$th partition in the rigged configuration, then the stretch $\tilde{\nu}^b$ corresponds to the set $\{L_i(l) | |L_i(l)| = b\}$. In other words, the stretch $\tilde{\nu}^b$ corresponds to the set of $l$-lanes of length $b$, or equivalently the set of columns of $\nu$ with height $b$.
\end{rem}

Let $R = (\nu_1,\nu_2,\ldots,\nu_n)$ be a rigged configuration we want to construct by our growth algorithm. To construct the compatible rigged partition $\nu_{i-1}$ given that we have already constructed $\nu_i, \nu_{i+1}, \ldots, \nu_n$, where the riggings of $\nu_i, \nu_{i+1}, \ldots, \nu_n$ are fixed, we will add noncontributing boxes and contributing boxes beneath the stretches of $\overline{\nu_i}$ (which has zero riggings by default). Roughly speaking, at most two rows of boxes will be added beneath each stretch of $\overline{\nu_i}$, with the first row consisting of noncontributing boxes and the second row consisting of contributing boxes. This is justified by Theorem \ref{constraints}. Of course, the contributing boxes added beneath each stretch of $\overline{\nu_i}$ must account for the riggings of $\nu_i$, by Lemma \ref{lane rig}. Before describing how to add boxes to $\overline{\nu_i}$, we need the notion of \textit{plateaus} to delineate the stretches of a rigged partition to which boxes can be added.

\subsection{Plateaus as Base for Construction}
\label{plateaus}
\begin{mydef}
We say that a cascading sequence $\beta$ (as well as its corresponding rigged configuration) is a \textbf{$(p,q,r)$-plateau} if it satisfies the following property: 
\begin{enumerate}
\item \label{noncon free} For every $i$ for which $L_i(p)$ exists, we have $|L_i(p-1)| = |L_i(p)|-1$ whenever  $|L_i(p-1)| < q$.
\item For every $(p-1)$-lane $L_i(p-1)$ with $|L_i(p-1)| \leq r$, $L_i(p-1)$ does not end at a right endpoint.
\item For any $k < p-1$ no $k$-lane ends at a right endpoint.
\end{enumerate}
If the above property holds for $q = \infty$ and $r = \infty$, then we call $\beta$ a \textbf{$p$-plateau}. If $\beta$ is an $m$-plateau for every $m \in [p]$, then we call $\beta$ a \textbf{$p^*$-plateau}; here (and in what follows) we will exclude the case $m=1$ from consideration, as the letter $0$ does not occur in $\beta$.
\end{mydef}

\begin{lem}
\label{mesa}
If $\beta$ is a $p^*$-plateau corresponding to the rigged configuration $(\mu_1,\mu_2,\ldots,$ $\mu_n)$, then $\mu_{l-1} = \overline{\mu_l}$ with zero riggings for $l \in [p]$.
\end{lem}
\begin{proof}
First we show that a right endpoint can only exist at the end of a lane. Suppose a right endpoint occurs in an $l$-lane $L_j(l)$, and let $I$ denote the lower subinterval containing this right endpoint. By the definition of cascading sequences, the only lower subintervals after $I$ containing $l$ as an entry must have $l$ as a right endpoint. However, any lower subinterval after $I$ that contains $l$ as a right endpoint must add its right endpoint to a lane $L_k(l)$ where $k>j$, by Lemma \ref{unif con}.

Since right endpoints can only occur at the end of a lane, we have $\mu_{l-1} = \overline{\mu_l}$ by the definition of $p^*$-plateau, and we have that $\mu_l$ has zero riggings by Lemma \ref{lane rig}, for $l \in [p]$
\end{proof}

\begin{exm}
\label{ex plat}
The cascading sequence consisting of lower subintervals $(7^1,8^1,9^1)$, $(8^2,9^2)$, $(8^3,9^3)$, $(9^4)$, $(6^1,7^1,8^1)$, $(7^2,8^2)$, $(8^4)$ is a $7^*$-plateau and a $(8,2,\infty)$-plateau. The cascading sequence consisting of the lower subintervals $(7^1,8^1,9^1,10^1)$, $(7^2,8^2,9^2,$ $10^2)$, $(8^3,9^3,10^3)$, $(8^4,9^4,10^4)$, $(8^5,9^5,10^5)$, $(6^1,7^1,8^1,9^1)$, $(7^3,8^2,9^2)$, $(5^1,6^1,7^1,8^1)$, $(6^2,7^2,8^2)$, $(6^3,7^3,8^3)$, $(5^2,$ $6^2,7^2)$ is a $7^*$-plateau and an $(8,2,2)$-plateau.
\end{exm}

\begin{rem}
As a start, notice that the cascading sequence consisting of the (singleton) lower subintervals $(n),(n),\ldots,(n)$ is an $n^*$-plateau.
\end{rem}

We now present procedures for adding at most two boxes to each column of the $(p-1)$st partition of a rigged configuration that is both a $(p,q,r)$-plateau and a $(p-1)^*$-plateau, generating all possible $(p-1)$st partitions compatible with the predetermined $p$th, $(p+1)$st, \ldots, $n$th partitions. Rigged configurations that are both a $(p,q,r)$-plateau and a $(p-1)^*$-plateau will serve as the ``skeletons'' upon which boxes are added in our growth algorithm.

\subsection{Adding Boxes to a Stretch}
\label{addbox}
Since any stretch $s$ of a partition $\lambda$ corresponds to all columns of some fixed height $\mathrm{ht}(s)$, we will also refer to $\mathrm{ht}(s)$ as the \textbf{height} of the stretch $s$.

\begin{conv}
\label{invisible}
For any $A_n$ rigged configuration $R' = (\nu_1,\nu_2,\ldots,\nu_n)$, by Remark \ref{bar rem} we already know that $\nu_{i-1} \supset \overline{\nu_i}$. If we label the stretches of $\nu_i$ from bottom to top by $g_1, g_2, \ldots, g_k$, and the stretches of $\overline{\nu_i}$ from bottom to top by $g'_1, g'_2, \ldots, g'_{k}$, then clearly $g'_i$ is identical to $g_i$ for $i \in [k-1]$. Here $g'_k$ is identical to $g_k$ if $|\nu^1_i| = |\nu^2_i|$, and is empty if $|\nu^1_i| > |\nu^2_i|$. In the case $|\nu^1_i| > |\nu^2_i|$, we will refer to $g'_k$ as an ``invisible'' stretch above the first row of $\overline{\nu_i}$, with $|g'_k| = |g_k|$. Thus, in either case, we will regard $\overline{\nu_i}$ as having identical copies of all the stretches of $\nu_i$; this will be convenient for when we talk about adding boxes to $\overline{\nu_i}$ to form $\nu_{i-1}$, where a box added beneath $g'_k$ will be in the first row of the resulting partition in the case $|\nu^1_i| > |\nu^2_i|$.
\end{conv}

Now, fix cascading sequence $\beta$ that is both a $(p-1)^*$-plateau and a $(p,q,r)$-plateau, with corresponding rigged configuration $R = (\mu_1,\mu_2,\ldots,\mu_n)$. We give two procedures for adding respectively noncontributing boxes and contributing boxes beneath the stretches $z_1, z_2, \ldots, z_a$ of $\mu_{p-1}$ (ordered from bottom to top, following Convention\ref{invisible}), which fixes $\mu_x$ for all $x > p$ and also fixes the shape of $\mu_p$ (though not necessarily the rigging, which will depend on the resulting $(p-1)$st partition). Assume that $\mathrm{ht}(z_1) < \mathrm{maxr}_{p-1}$; otherwise no boxes can be added beneath $z_1$. Both procedures output both the desired cascading sequence and its corresponding rigged configuration, and can be applied repeatedly to add boxes to multiple stretches sequentially.

\begin{proc}[Adding Noncontributing Boxes to a Given Stretch]
\label{addnoncon}
Suppose $\mathrm{ht}(z_i) < \min(q,r)$. The following algorithm adds $n_i$ noncontributing boxes beneath $z_i$, where $0 \leq n_i \leq |z_i|$:

We will add $n_i$ copies of $p-1$, $n_i$ copies of $p-2$, $\ldots$, $n_i$ copies of $p-(\mathrm{ht}(z_i)+1)$ to $\beta$ as follows. Let $B_{v}$ denote the set of lower subintervals with head $v$. We will delete a number of elements from each $B_v$ before lengthening the first $n_i$ of the remaining elements.

Label from left to right by $\breve{I}_1^{p-\mathrm{ht}(z_i)-1}$, $\breve{I}_2^{p-\mathrm{ht}(z_i)-1}$, $\ldots$, $\breve{I}_w^{p-\mathrm{ht}(z_i)-1}$ the elements of $B_{p-\mathrm{ht}(z_i)-1}$. Then delete the element $\breve{I}_1^{p-\mathrm{ht}(z_i)} \in B_{p-\mathrm{ht}(z_i)}$ nearest $\breve{I}_1^{p-\mathrm{ht}(z_i)-1}$ left of $\breve{I}_1^{p-\mathrm{ht}(z_i)-1}$, delete the element $\breve{I}_2^{p-\mathrm{ht}(z_i)} \in B_{p-\mathrm{ht}(z_i)}$ nearest $\breve{I}_2^{p-\mathrm{ht}(z_i)-1}$ left of $\breve{I}_2^{p-\mathrm{ht}(z_i)-1}$, $\ldots$, delete the element $\breve{I}_w^{p-\mathrm{ht}(z_i)} \in B_{p-\mathrm{ht}(z_i)}$ nearest $\breve{I}_w^{p-\mathrm{ht}(z_i)-1}$ left of $\breve{I}_w^{p-\mathrm{ht}(z_i)-1}$. Let $\overline{B}_{p-\mathrm{ht}(z_i)}$ denote the subset obtained from $B_{p-\mathrm{ht}(z_i)}$ after performing this sequence of deletions. Next, delete the element $\breve{I}_1^{p-\mathrm{ht}(z_i)+1} \in B_{p-\mathrm{ht}(z_i)+1}$ nearest $\breve{I}_1^{p-\mathrm{ht}(z_i)}$ left of $\breve{I}_1^{p-\mathrm{ht}(z_i)}$, delete the element $\breve{I}_2^{p-\mathrm{ht}(z_i)+1} \in B_{p-\mathrm{ht}(z_i)+1}$ nearest $\breve{I}_2^{p-\mathrm{ht}(z_i)}$ left of $\breve{I}_2^{p-\mathrm{ht}(z_i)}$, $\ldots$, delete the element $\breve{I}_w^{p-\mathrm{ht}(z_i)+1} \in B_{p-\mathrm{ht}(z_i)+1}$ nearest $\breve{I}_w^{p-\mathrm{ht}(z_i)}$ left of $\breve{I}_w^{p-\mathrm{ht}(z_i)}$. Let $\overline{B}_{p-\mathrm{ht}(z_i)+1}$ denote the subset obtained from $B_{p-\mathrm{ht}(z_i)+1}$ after performing this sequence of deletions. In general, delete the element $\breve{I}_1^{p-\mathrm{ht}(z_i)+c} \in B_{p-\mathrm{ht}(z_i)+c}$ nearest $\breve{I}_1^{p-\mathrm{ht}(z_i)+c-1}$ left of $\breve{I}_1^{p-\mathrm{ht}(z_i)+c-1}$, delete the element $\breve{I}_2^{p-\mathrm{ht}(z_i)+c} \in B_{p-\mathrm{ht}(z_i)+c}$ nearest $\breve{I}_2^{p-\mathrm{ht}(z_i)+c-1}$ left of $\breve{I}_2^{p-\mathrm{ht}(z_i)+c-1}$, $\ldots$, delete the element $\breve{I}_w^{p-\mathrm{ht}(z_i)+c} \in B_{p-\mathrm{ht}(z_i)+c}$ nearest $\breve{I}_w^{p-\mathrm{ht}(z_i)+c-1}$ left of $\breve{I}_w^{p-\mathrm{ht}(z_i)+c-1}$. Let $\overline{B}_{p-\mathrm{ht}(z_i)+c}$ denote the subset obtained from $B_{p-\mathrm{ht}(z_i)+c}$ after performing this sequence of deletions.

Finally, lengthen the first $n_i$ elements (in left-right order as usual) of $\overline{B}_{p-\mathrm{ht}(z_i)+d}$ in $\beta$, for $d = 0, 1, \ldots, \mathrm{ht}(z_i)$. 
\end{proc}
\begin{notn}
Let $C_{\mathrm{ht}(z_i)}$ denote the set of lower subintervals $\breve{I}_j^{p-\mathrm{ht}(z_i)+c}$ of $\beta$ deliberately fixed (i.e. not lengthened) in Procedure \ref{addnoncon}. We will call $C_{\mathrm{ht}(z_i)}$ the set of \textbf{deleted elements} of $\beta$, or the set of \textbf{fixed elements} of $\beta$.
\end{notn}
\begin{rem}
In Procedure \ref{addnoncon}, we say that the element $\breve{I}_u^{p-\mathrm{ht}(z_i)+c}$ is \textbf{paired} with the element $\breve{I}_u^{p-\mathrm{ht}(z_i)+c-1}$, for each $u \in [w]$. This pairing process used to obtain $C_{\mathrm{ht}(z_i)}$ is in fact the same pairing/bracketing process in the definition of the Kashiwara operator. 

A rough illustration of the pairing/bracketing in Procedure \ref{addnoncon}: If we let $a$ denote a lower subinterval with head $j$ and $b$ denote a lower subinterval with head $j-1$ ($a$ and $b$ are used as shorthand here; the $a$'s (resp. $b$'s) are not necessarily identical), and if we let $aaabbaabbabaa$ be the cascading subsequence whose lower subintervals have only $j$ or $j-1$ as head, then the pairing and deletion process in Procedure \ref{addnoncon} works as follows.

$aaabbaabbabaa \rightarrow aa(ab)ba(ab)b(ab)aa \rightarrow a(ab)(ab)aa \rightarrow aaa$ (this means that only the remaining lower subintervals $aaa$ can be lengthened)
\end{rem}

That Procedure \ref{addnoncon} works as described will be proven in Section \ref{proof}. Meanwhile, let us look at some examples of how this procedure works.

\begin{exm}
Consider the cascading sequence $\alpha$ (with lanes marked by superscripts as usual) consisting of the lower subintervals $(8^1,9^1,10^1,11^1)$, $(8^2,9^2,10^2,11^2)$, $(8^3,9^3,10^3,11^3)$, $(7^1,8^1,$ $9^1,10^1)$, $(7^2,8^2,9^2,10^2)$, $(8^4,9^4,10^4)$, $(6^1,7^1,8^1,9^1)$, $(7^3,8^3,9^3)$, which is an $8^*$-plateau. We have $$\mu_7 = \begin{array}[t]{r|c|c|c|l}
 \cline{2-4} 0 &\phantom{|}&\phantom{|}&\phantom{|}& 0 \\
 \cline{2-4} 0 &\phantom{|}& \multicolumn{3 }{l}{ 0 } \\
 \cline{2-2} 
\end{array}.$$
To add two noncontributing boxes to the second stretch of $\mu_7$, we add two copies of $7$ and two copies of $6$ to $\alpha$; $C_1$ in this case consists of the lower subintervals $(6^1,7^1,8^1,9^1)$, $(7^2,8^2,9^2,10^2)$, $(8^3,9^3,10^3,11^3)$. The resulting cascading sequence $\alpha'$ (where the added copies are in bold) consisting of the lower subintervals $(\textbf{7}^1,8^1,9^1,10^1,11^1)$, $(\textbf{7}^2,8^2,9^2,10^2,11^2)$, $(8^3,9^3,10^3,11^3)$, $(\textbf{6}^1,7^1,8^1,9^1,10^1)$, $(7^3,8^2,9^2,10^2)$, $(8^4,9^4,10^4)$, $(6^2,7^2,8^1,9^1)$, $(\textbf{6}^3,7^3,8^3,9^3)$ corresponds to the resulting rigged configuration whose seventh partition is 
\begin{array}[t]{r|c|c|c|l}
 \cline{2-4} -2 &\phantom{|}&\phantom{|}&\phantom{|}& 0 \\
 \cline{2-4}  &\phantom{|}&\phantom{|}&\phantom{|}& 0 \\
 \cline{2-4} 
\end{array}.
\end{exm}

\begin{exm}
\label{examp}
Consider the cascading sequence $\alpha$ (with lanes marked by superscripts as usual) consisting of the lower subintervals $(7^1,8^1,9^1,10^1,11^1)$, $(8^2,9^2,10^2,$ $11^2)$, $(8^3,9^3,10^3,11^3)$, $(6^1,$ $7^1,8^1,9^1,10^1)$, $(7^2,8^2,9^2,10^2)$, $(8^4,9^4,10^4)$, $(6^2,7^2,8^1,9^1)$, $(7^3,8^3,9^3)$, which is a $7^*$-plateau and an $(8,2,3)$-plateau. We have $$\mu_7 = \begin{array}[t]{r|c|c|c|l}
 \cline{2-4} -1 &\phantom{|}&\phantom{|}&\phantom{|}& 0 \\
 \cline{2-4} -1 &\phantom{|}&\phantom{|}& \multicolumn{2 }{l}{ 0 } \\
 \cline{2-3} 
\end{array}$$ and $$\mu_8 = \begin{array}[t]{r|c|c|c|c|l}
 \cline{2-5} -2 &\phantom{|}&\phantom{|}&\phantom{|}&\phantom{|}& 0 \\
 \cline{2-5} -2 &\phantom{|}&\phantom{|}&\phantom{|}& \multicolumn{2 }{l}{ 0 } \\
 \cline{2-4} -1 &\phantom{|}& \multicolumn{4 }{l}{ 0 } \\
 \cline{2-2} 
\end{array}.$$
To add one noncontributing box to the third stretch (which has length $1$, since $|\mu_8^1| - |\mu_7^1| = 1$) of $\mu_7$, we add one copy of $7$ to $\alpha$; $C_0$ in this case consists of the lower subintervals $(7^2,8^2,9^2,10^2)$, $(7^3,8^3,9^3)$, $(8^3,9^3,10^3,11^3)$, $(8^4,9^4,10^4)$. The resulting cascading sequence $\alpha'$ (where the added copies are in bold) consisting of the lower subintervals $(7^1,8^1,9^1,10^1,11^1)$, $(\textbf{7}^2,8^2,9^2,10^2,11^2)$, $(8^3,9^3,10^3,11^3)$, $(6^1,7^1,8^1,9^1,10^1)$, $(7^3,8^2,9^2,10^2)$, $(8^4,9^4,10^4)$, $(6^2,7^2,8^1,9^1)$, $(7^4,8^3,9^3)$ corresponds to the resulting rigged configuration whose seventh partition is $\begin{array}[t]{r|c|c|c|c|l}
 \cline{2-5} -2 &\phantom{|}&\phantom{|}&\phantom{|}&\phantom{|}& 0 \\
 \cline{2-5} -1 &\phantom{|}&\phantom{|}& \multicolumn{3 }{l}{ 0 } \\
 \cline{2-3} 
\end{array}$.
\end{exm}

\begin{exm}
Consider the cascading $11$-sequence consisting of the lower subintervals $(8^1,9^1,$ $10^1,11^1),$ $(8^2,9^2,10^2,11^2),$ $(8^3,9^3,10^3,11^3),$ $(8^4,9^4,10^4,11^4),$ $(7^1,8^1,9^1,10^1),$ $(7^2,8^2,9^2,10^2),$ $(7^3,8^3,9^3,10^3),$ $(8^5,9^5,10^5),$ $(6^1,7^1,8^1,9^1),$ $(6^2,7^2,8^2,9^2),$ $(6^3,7^3,8^3,9^3),$ $(7^4,8^4,9^4),$ $(7^5,8^5,9^5),$ $(5^1,6^1,7^1,8^1),$ $(6^4,7^4,8^4)$. \\
To add boxes to the stretch of height two, notice that $C_2$ consists of the lower subintervals $(5^1,6^1,7^1,8^1)$, $(6^3,7^3,8^3,9^3)$, $(7^3,8^3,9^3,10^3)$, $(8^4,9^4,10^4,11^4)$. If we add three boxes to the stretch of height two, we obtain $(\textbf{7}^1,8^1,9^1,10^1,11^1),$ $(\textbf{7}^2,8^2,9^2,10^2,11^2),$ $(\textbf{7}^3,8^3,9^3,10^3,11^3),$ $(8^4,9^4,10^4,11^4),$ $(\textbf{6}^1,7^1,8^1,9^1,10^1)$, $(\textbf{6}^2,7^2,8^2,9^2,10^2)$, $(7^4,8^3,9^3,10^3)$, $(8^5,9^5,10^5)$, $(\textbf{5}^1,6^1,7^1,$ $8^1,9^1)$, $(\textbf{5}^2,6^2,7^2,8^2,9^2)$, $(6^3,7^3,8^3,9^3)$, $(\textbf{6}^4,7^4,8^4,9^4)$, $(7^5,8^5,9^5)$, $(5^3,6^3,7^3,8^1)$, $(\textbf{5}^4,6^4,7^4,8^4)$. \\
To add boxes to the stretch of height one, notice that $C_1$ consists of the lower subintervals $(6^3,7^3,8^3,9^3)$, $(6^4,7^4,8^4,9^4)$, $(6^1,7^1,8^1,9^1,10^1)$, $(6^2,7^2,8^2,9^2,10^2)$, $(7^4,8^3,9^3,10^3)$, $(7^1,8^1,9^1,$ $10^1,11^1),$ $(7^2,8^2,9^2,10^2,11^2),$ $(7^3,8^3,9^3,10^3,11^3)$, $(8^4,9^4,10^4,11^4)$. If we add one box to the stretch of height one, we obtain $(7^1,8^1,9^1,10^1,11^1),$ $(7^2,8^2,9^2,10^2,11^2),$ $(7^3,8^3,9^3,10^3,11^3),$ $(8^4,9^4,10^4,11^4),$ $(6^1,7^1,8^1,9^1,10^1)$, $(6^2,7^2,8^2,9^2,10^2)$, $(7^4,8^3,9^3,10^3)$, $(\textbf{7}^5,8^5,9^5,10^5)$, $(5^1,6^1,$ $7^1,8^1,9^1)$, $(5^2,6^2,7^2,8^2,9^2)$, $(6^3,7^3,8^3,9^3)$, $(6^4,7^4,8^4,9^4)$, $(\textbf{6}^5,7^5,8^5,9^5)$, $(5^3,6^3,7^3,8^1)$, $(5^4,6^4,$ $7^4,8^4)$.
\end{exm}

\begin{proc}[Adding Contributing Boxes to a Given Stretch]
\label{addcon}
Suppose $\mathrm{ht}(z_i)$ $<$ $r$.
\begin{enumerate}
\item \label{addcon1} To add $n_i$ contributing boxes beneath $z_i$, where $0 \leq n_i \leq |z_i|$: Add $n_i$ copies of the lower subinterval $(p-\mathrm{ht}(z_i)-1,p-\mathrm{ht}(z_i),\ldots,p-1)$ to the right of $\beta$. 
\item \label{addcon2} Exceptional Case: We can add any number $N \in \mathbb{Z}_{\geq0}$ of contributing boxes to the top row of $\mu_{p-1}$ by adding $N$ singleton lower subintervals $(p-1)$ to the right of $\beta$.
\end{enumerate}
\end{proc}
That Procedure \ref{addcon} works as described will be proven in Section \ref{proof}. 

\subsection{Proof of the Procedures for Adding Boxes}
\label{proof}
We now show that the two procedures for adding boxes to a stretch works as described.

\begin{lem}
\label{plateaudet}
Suppose $\alpha$ is a $(p,q,r)$-plateau and a $(p-1)^*$-plateau, with corresponding rigged configurations $R' = (\nu_1,\ldots,\nu_{p-1},\nu_p,\ldots,\nu_n)$. Then $R'$ (and hence $\alpha$) is completely determined once we know the rigged partitions $\nu_{p-1},\nu_p,\ldots,\nu_n$.
\end{lem}
\begin{proof}
Follows from Lemma \ref{mesa}.
\end{proof}
\begin{rem}
Hence, if we keep $\nu_p,\ldots,\nu_n$ fixed, the range of all possible such $\alpha$ is completely determined by the range of all possible $\nu_{p-1}$ (obtained by adding boxes to the allowed stretches of $\overline{\nu}_p$).
\end{rem}

For any cascading sequence $\alpha$ with corresponding rigged configuration $R$, let $\alpha[i]$ denote the sub-cascading sequence formed by the first $i$ lower subintervals of $\alpha$, and let $R[i]$ denote the rigged configuration corresponding to $\alpha[i]$; we call $\alpha[i]$ the \textbf{initial $i$-segment of $\alpha$}. For any lower subinterval $I$ of $\alpha$, let $\alpha_I$ denote the portion of $\alpha$ preceding $I$, and let $R_I$ denote the rigged configuration corresponding to $\alpha_I$. The following lemmas show that, if $\alpha$ is an $l$-plateau, then it has certain nice properties, which we will use in the proof of the main lemma of this section.

\begin{lem}
\label{auxiliary}
Suppose that $\alpha$ is a cascading sequence with corresponding rigged configuration $R$ such that $R[i] = (\mu_1,\mu_2,\ldots,\mu_n)$ has the property that the $j$th column of $\overline{\mu_k}$ is shorter than the $j$th column of $\mu_{k-1}$ for all $j \geq b$ for some $b$. Then $R[i+1] = (\mu'_1,\mu'_2,\ldots,\mu'_n)$ has the property that the $j$th column of $\overline{\mu'_k}$ is shorter than the $j$th column of $\mu'_{k-1}$ for all $j \geq c$ for some $c \geq b$.
\end{lem}
\begin{rem}
This lemma is used to prove Lemma \ref{no head}.
\end{rem}
\begin{proof}
Let $I$ denote the lower subinterval after $\alpha[i]$. We may assume that $I$ has head at most $k$. If $I$ has head $k$, then it adds a box to the first row of $\mu_k$, and the conclusion is still true for $c = b$. Suppose $I$ has head less than $k$. Suppose $I$ adds a box to the $p$th column of $\mu_{k-1}$. Then it must add a box to the $q$th column of $\mu_k$, where $p \geq q$. If $q < b$, then the conclusion is still true for $c = b$. Suppose that $q \geq b$. Then the conclusion is true for $c = p$.
\end{proof}

\begin{lem}
\label{no head}
Suppose that $\alpha$ is an $l$-plateau with corresponding rigged configuration $R = (\lambda_1,\lambda_2,\ldots,\lambda_n)$. Then no lower subinterval of $\alpha$ with head less than $l$ can contain the head of some $l$-lane.
\end{lem}
\begin{proof}

Suppose for a contradiction that $\alpha$ has a lower subinterval $J$ with head less than $l$ containing the head of some $l$-lane. Without loss of generality assume that $J$ is the leftmost such lower subinterval. Let $\alpha[i]$ denote the portion of $\alpha$ preceding $J$, and let $R[i] = (\mu_1,\mu_2,\ldots,\mu_n)$ denote the corresponding rigged configuration. By definition, $J$ adds a box to the top row of $\mu_l$. By Lemma [18], we have $\mu_{l-1} \subset \mu_l$, so $J$ must add a box to the top row of $\mu_{l-1}$ as well; in fact, the top rows of $\mu_{l-1}$ and $\mu_l$ must be identical. Then, again by Lemma [18], $R[i+1] = (\nu_1,\nu_2,\ldots,\nu_n)$ has the property that the $j$th column of $\overline{\nu_l}$ is shorter than the $j$th column of $\nu_{l-1}$ for all $j \geq m$, where $m$ is the length of the top row of $\nu_{l-1}$. Repeatedly applying Lemma \ref{auxiliary}, we conclude that the $d$th column of $\overline{\lambda_l}$ is shorter than the $d$th column of $\lambda_{l-1}$, for some $d$. This contradicts the $l$-plateau assumption on $\alpha$.
\end{proof}

\begin{lem}
\label{inherit runs}
Suppose that $\alpha$ is an $l$-plateau. Then the following hold: 
\begin{enumerate}
\item $\alpha[i]$ is also an $l$-plateau for any $i$.
\item For any lower subinterval $I$ of $\alpha[i]$, the entries $l-1$ and $l$ of $I$ have the same lane number.
\end{enumerate}
\end{lem}
\begin{proof}
Notice that, by definition, $I$ must contain $l$ whenever it contains an entry less than $l$. The two items are vacuously true if every lower subinterval of $\alpha[i]$ has head greater than $l$. In the base case that $\alpha[i]$ ends in $J$, where $J$ is the first lower subinterval of $\alpha$ with head at most $l$, $J$ must contain the head of some $l$-lane, so $J$ must have head $l$ by Lemma \ref{no head}, and hence the two items hold for $\alpha[i]$. Now we suppose that the two items hold for $\alpha[i]$, and prove them for $\alpha[i+1]$. Let $I$ denote the lower subinterval following $\alpha[i]$. If $I$ has head at least $l$, clearly the two items hold true (vacuously true for the second item). Suppose $I$ has head less than $l$. By Lemma \ref{no head}, $I$ must add a box beneath the top row of the $l$th partition. Since the first item holds for $\alpha[i]$, the $(l-1)$st partition is exactly the portion of the $l$th partition beneath the first row, in the corresponding rigged configuration. If $I$ adds a box to row $r^{l-1}$ of the $(l-1)$st partition (where $r^{l-1}$ is the uppermost row of its length), then it must add a box to row $r^l$ of the $l$th partition, where $r^l$ is the uppermost row of the $l$th partition with length $|r^{l-1}|$. This shows that $\alpha[i+1]$ is still an $l$-plateau, and that the entries $l-1$ and $l$ of $I$ have the same lane number $|r^{l-1}|+1$, completing the induction.
\end{proof}

\begin{lem}
\label{lyndon}
Suppose that $\alpha$ is an $l$-plateau. For any $i$ and any $j < l$, $\alpha[i]$ has no fewer lower subintervals with head $l$ than lower subintervals with head $j$. We say that $\alpha$ satisfies the \textbf{Lyndon property} for letter $l$.
\end{lem}
\begin{rem}
In particular, $\alpha$ can be considered as a left Lyndon word in the letters $l$ and $j$.
\end{rem}
\begin{proof}
Fix $i$ and $j < l$. By Lemma \ref{inherit runs}, $\alpha[i]$ is also an $l$-plateau. Let $R[i] = (\nu_1,\nu_2,\ldots,\nu_n)$ denote the rigged configuration corresponding to $\alpha[i]$. By Lemma \ref{snake} and by the definition of an $l$-plateau, we have $\nu_j \subset \nu_{j+1} \subset \ldots \subset \nu_l$. By Lemma \ref{no head}, any lower subinterval that adds a box to the first row of the $l$th partition must have head $l$. Since any lower subinterval of $\alpha[i]$ with head $j$ adds a box to the top row of the $j$th partition and any lower subinterval of $\alpha[i]$ with head $l$ adds a box to the top row of the $l$th partition, it follows that $\alpha[i]$ has no fewer lower subintervals with head $l$ than lower subintervals with head $j$.
\end{proof}

\begin{lem}
\label{lyndonimplies}
Suppose that $\alpha$ satisfies the Lyndon property for all letters at most $l$. If $I$ is a lower subinterval of $\alpha$ with head $j \leq l$, then all entries of $I$ at most $l$ have the same lane number, and the depth of entry $l'$ is $l'-j+1$ for any $j \leq l' \leq l$.
\end{lem}
\begin{proof}
Claim is obvious for the base case of the initial segment $\alpha[i_1]$ ending in the first lower subinterval with head $l$. Now suppose that $I$ is a lower subinterval of $\alpha$ with head $j \leq l$, and suppose that the claim holds for $\alpha_I$. We show that $I$ satisfies the desired properties. Let $I'$ denote the lower subinterval of $\alpha_I$ with head $j$ nearest $I$; if $I'$ does not exist, then $\alpha_I$ has no lower subinterval with head smaller than $j+1$, so all entries of $I$ at most $l$ have lane number one and the claim follows immediately (by the inductive hypothesis, the depth of entry $c \leq l$ of $I$ is one more than the depth of entry $c$ of a lower subinterval with head $j+1$). By the inductive hypothesis, all entries of $I'$ at most $l$ have the same lane number $k$, and the depth of entry $l'$ is $l'-j+1$ for any $j \leq l' \leq l$. Fix $j \leq l' \leq l$. We show that entry $l'$ of $I$ has lane number $k+1$ and depth $l'-j+1$. Let $I''$ denote any lower subinterval of $\alpha_I$ after $I'$. By definition, $I''$ has head $j'' \neq j$. Let $j \leq d \leq l'$. If $j'' < j$, then entry $d$ of $I''$ has depth greater than that of entry $d$ of $I'$ by inductive hypothesis. If $j'' > j$, then entry $d$ of $I''$ has depth less than that of entry $d$ of $I'$ by inductive hypothesis. It follows that the number of $d$-lanes of length at least $d-j+1$ in $\alpha_I$ is $k$; equivalently, the $d$th partition of the rigged configuration corresponding to $\alpha_I$ has exactly $k$ columns of height at least $d-j+1$. Recall that $d-j+1$ is the depth of entry $d$ of $I'$. Moreover, since $\alpha_I$ has more lower subintervals with head $j+1$ than those with head $j$ by the Lyndon property, $\alpha_I$ has some $d$-lanes of length $d-(j+1)+1 = d-j$. Since the head $j$ of $I$ clearly has lane number $k+1$, we deduce that the entry $d$ of $I$ must also have lane number $k+1$ and depth $d-j+1$, for $d = j, j+1, \ldots, l'$ in that order, by Lemma \ref{snake}. This completes the induction.
\end{proof}

We now complete the proof of the two procedures for adding boxes. If a cascading sequence $\gamma$ has entry $g$ and lanes $L, L'$ such that $L' = L \oplus (g)$, we say that $L'$ is the \textbf{lengthening} of $L$, and that we \textbf{lengthen} $L$ to obtain $L'$.
\begin{lem}[Main Lemma]
\label{addboxproof}
Let $\beta$ be both a $(p-1)^*$-plateau and a $(p,q,r)$-plateau, with corresponding rigged configuration $R = (\mu_1,\mu_2,\ldots,\mu_n)$. Label the stretches of $\mu_{p-1}$ from bottom to top as $z_1, z_2, \ldots, z_a$. Then the following is true.
\begin{enumerate}
\item \label{addncworks} Suppose $\mathrm{ht}(z_i) < \min(q,r)$. Let $\beta^@$ denote the cascading sequence obtained from $\beta$ via Procedure \ref{addnoncon}. $\beta^@$ thus obtained is a $(p-1)^*$-plateau and a $(p,\mathrm{ht}(z_{i}),r)$-plateau, and $\beta^@$ corresponds to the rigged configuration obtained after adding $n_i$ noncontributing boxes beneath $z_i$, where $0 \leq n_i \leq |z_i|$, and fixing $\mu_x$ for all $x \geq p$.
\item \label{headrestr} Suppose $\mathrm{ht}(z_i) < \min(q,r)$, and let $0 < a' \leq \mathrm{ht}(z_i)$ be an integer. 
\begin{enumerate}
\item There exists $j_1$ such that $\beta[j_1]$ contains all the lower subintervals with head $p-1$ containing the head of some $p$-lane, and that $\beta[j_1]$ contains no more lower subintervals with head $p-1$ than lower subintervals with head $p-2$. Let such $j_1$ be minimal. Then there exists $j_2 \geq j_1$ such that $\beta[j_2]$ contains no more lower subintervals with head $p-2$ than lower subintervals with head $p-3$. Let such $j_2$ be minimal. Then there exists $j_3 \geq j_2$ such that $\beta[j_3]$ contains no more lower subintervals with head $p-3$ than lower subintervals with head $p-4$. This continues until, for minimal $j_{a'-1}$ there exists $j_{a'} \geq j_{a'-1}$ such that $\beta[j_{a'}]$ contains no more lower subintervals with head $p-a'$ than lower subintervals with head $p-a'-1$.
\item If $I$ is an element of $B_{p-h}$ outside $\beta[j_{p-(p-h)}] = \beta[j_{h}]$ where $1 \leq h \leq \mathrm{ht}(z_i)$, then $I$ has entry $p$ of depth $p-(p-h)+1 = h+1$, entry $p-1$ of depth $p-1-(p-h)+1 = h$, and one common lane number for all entries not exceeding $p$. 
\item If $\widehat{I}$ is an element of $B_{p-1}$ containing the head of some $p$-lane, then all elements of $B_p$ in $\beta_{\widehat{I}}$ are elements of $C_{\mathrm{ht}(z_i)}$. 
\end{enumerate}
\item \label{addconworks} Suppose $\mathrm{ht}(z_i)$ $<$ $r$. 
\begin{enumerate}
\item Suppose $\beta^!$ is the cascading sequence obtained from $\beta$ via the Procedure \ref{addcon}(\ref{addcon1}). $\beta^!$ thus obtained is a $(p-1)^*$-plateau and a $(p,q,\mathrm{ht}(z_{i}))$-plateau, and $\beta^!$ corresponds to the rigged configuration obtained after adding $n_i$ contributing boxes beneath $z_i$ in $\mu_{p-1}$, where $0 \leq n_i \leq |z_i|$, and fixing $\mu_x$ for all $x > p$ as well as fixing the shape of $\mu_p$. 
\item Suppose $\beta^!$ is the cascading sequence obtained from $\beta$ via the Procedure \ref{addcon}(\ref{addcon2}). Then $\beta^!$ thus obtained is a $(p-1)^*$-plateau and a $(p,q,0)$-plateau, and corresponds to the rigged configuration obtained after adding any number $N \in \mathbb{Z}_{\geq0}$ of contributing boxes to the top row of $\mu_{p-1}$, and fixing $\mu_x$ for all $x > p$ as well as fixing the shape of $\mu_p$.
\end{enumerate} 
\end{enumerate}
\end{lem}
\begin{rem}
\label{addbox rem}
Keep the following in mind for the proof that follows.
\begin{enumerate}
\item Item \ref{headrestr} is a technical fact about $\beta$ that will be used in the proof of Procedure \ref{addnoncon} (i.e. Item \ref{addncworks}) in an induction argument.
\item Although adding boxes beneath $z_i$ changes the partition, it does not change the stretches $z_x$ for any $x>i$, so we will continue referring to the stretches $z_x$ even though they may well belong to a partition different from $\mu_{p-1}$.
\item Convention: Let $\alpha, \alpha'$ be cascading sequences with corresponding rigged configurations $S, S'$ respectively. We say that $\alpha[i]$ and $\alpha'[i]$ \textit{have the same $l$-lanes} or \textit{have identical $l$-lanes} if the following holds: $l$ appears as an entry in $\alpha[i]$ the same number of times as $l$ appears as an entry in $\alpha'[i]$, and the $j$th occurrence of $l$ in $\alpha[i]$ has the same lane number as the $j$th occurrence of $l$ in $\alpha'[i]$. In particular, this implies that the $l$th partitions of $S[i]$ and $S'[i]$ are identical.
\end{enumerate}
\end{rem}
\begin{proof}
We prove these items by induction on $i$ (i.e. one stretch at a time by decreasing height); observe that conditions on $\beta$ become less restrictive with smaller $q$ and $r$, while the number of stretches beneath which boxes can be added decreases. In the base case where $\beta$ is a $p^*$-plateau, Item \ref{headrestr} holds vacuously with $j_1 = 0$, by Lemma \ref{lyndonimplies}.


In the general case, suppose $\beta$ is both a $(p-1)^*$-plateau and a $(p,q,r)$-plateau. We will prove Item \ref{addconworks}, executability of Procedure \ref{addnoncon}, Item \ref{addncworks}, and finally Item \ref{headrestr}, in that order. \\

\textit{Proof of Item \ref{addconworks}} \\
We first prove Item \ref{addconworks}, that Procedure \ref{addcon} works as stated. Suppose $\mathrm{ht}(z_i)$ $<$ $r$. By Lemma \ref{lyndonimplies}, if $I$ is a lower subinterval whose entry $p-1$ has depth $\mathrm{ht}(z_{i-1})$, then $I \in B_{p-\mathrm{ht}(z_{i-1})}$. Since $p-\mathrm{ht}(z_{i-1}) \leq p-\mathrm{ht}(z_i)-1$, and since $\beta$ is a $(p-1)^*$-plateau, the entry $p-1$ of any of the added $(p-\mathrm{ht}(z_i)-1,p-\mathrm{ht}(z_i),\ldots,p-1)$ has depth not exceeding $p-1 - (p-\mathrm{ht}(z_{i-1})) +1 = \mathrm{ht}(z_{i-1})$. Notice that $\beta$ has $|B_{p-\mathrm{ht}(z_i)}| = |B_{p-\mathrm{ht}(z_{i-1})}| + |z_i|$ and $|B_{p-\mathrm{ht}(z_i)-1}| = |B_{p-\mathrm{ht}(z_{i-1})}|$ by definition. Since $|B_{p-\mathrm{ht}(z_i)}| = |B_{p-\mathrm{ht}(z_i)-1}| + |z_i|$, and since $\beta$ is a $(p-1)^*$-plateau, the entry $p-1$ of any of the added $(p-\mathrm{ht}(z_i)-1,p-\mathrm{ht}(z_i),\ldots,p-1)$ has depth exceeding $p-1 - (p-\mathrm{ht}(z_i)) +1 = \mathrm{ht}(z_i)$. By Lemma \ref{unif con}, the entries $p-1$ of the $n_i$ added copies of $(p-\mathrm{ht}(z_i)-1,p-\mathrm{ht}(z_i),\ldots,p-1)$ must occupy $n_i$ distinct columns, so it follows that these entries $p-1$ must have depth $\mathrm{ht}(z_i)+1$, as desired. Clearly, $\beta^!$ is a $(p-1)^*$-plateau as well.

That Part \ref{addcon2} of this procedure works for adding boxes to the top row is obvious from Lemma \ref{snake}. Finally, the property of Item \ref{headrestr} is clearly preserved by Procedure \ref{addcon}. This concludes our proof of Item \ref{addconworks}. \\

Now assume that $\mathrm{ht}(z_i) < \min(q,r)$, for which Item \ref{headrestr} holds for $\beta$. We show that Items \ref{addncworks} and \ref{headrestr} hold for $\beta^@$. Let $\widehat{B}_{p-1}$ denote the set of elements of $B_{p-1}$ containing the head of some $p$-lane. Let $B^@_v$ denote the number of lower subintervals of $\beta^@$ with head $v$. Let $\widehat{B}^@_{p-1}$ denote the set of elements of $B^@_{p-1}$ containing the head of some $p$-lane. 

\textit{Proof of Executability} \\
We first prove that Procedure \ref{addnoncon} is executable. By Lemma \ref{lyndon}, it immediately follows that the deletions (recall that a deleted lower subinterval is ultimately fixed by the procedure) and lengthening specified in Procedure \ref{addnoncon} are executable for all pairs $B_{p-j}, B_{p-j-1}$ for all $j \geq 1$. More precisely, by Lemma \ref{lyndonimplies}, after deleting all the elements of $C_{\mathrm{ht}(z_i)}$ from $B_{p-\mathrm{ht}(z_i)}$, there will be exactly $|z_i|$ elements of $B_{p-\mathrm{ht}(z_i)}$ remaining; after deleting all the elements of $C_{\mathrm{ht}(z_i)}$ from $B_{p-\mathrm{ht}(z_i)+1}$, there will be at least $|z_i|$ elements of $B_{p-\mathrm{ht}(z_i)+1}$ remaining; after deleting all the elements of $C_{\mathrm{ht}(z_i)}$ from $B_{p-\mathrm{ht}(z_i)+2}$, there will be at least $|z_i|$ elements of $B_{p-\mathrm{ht}(z_i)+2}$ remaining; and so on. 

We now verify that the deletion and lengthening are executable for the pair $B_{p}, B_{p-1}$. By Item \ref{headrestr} and Lemma \ref{lyndon}, all elements of $\widehat{B}_{p-1}$ are elements of $C_{\mathrm{ht}(z_i)}$. By Lemma \ref{lyndonimplies}, no element of $B_l$ contains the head of some $p$-lane or the head of some $(p-1)$-lane, for any $l < p-1$. Hence only the elements of $\widehat{B}_{p-1}$ and $B_{p}$ contain the head of some $p$-lane. It follows that any initial segment of $\beta$ contains no fewer elements of $B_p$ than elements of $B_{p-1}-\widehat{B}_{p-1}$; if an initial segment contained fewer elements of $B_p$ than elements of $B_{p-1}-\widehat{B}_{p-1}$, then some of the latter elements would have to belong to $\widehat{B}_{p-1}$, a contradiction. 

Let $\tilde{I} \in \widehat{B}_{p-1}$. By Lemma \ref{snake}, $\beta_{\tilde{I}}$ has the same number of $p$-lanes and $(p-1)$-lanes. By Item \ref{headrestr}, every $I_1 \in B_p$ in $\beta_{\tilde{I}}$ belongs to $C_{\mathrm{ht}(z_i)}$. If $\tilde{I}$ is the leftmost element of $\widehat{B}_{p-1}$, then clearly every $I_1 \in B_p$ in $\beta_{\tilde{I}}$ must be paired with an $I'_1 \in B_{p-1}-\widehat{B}_{p-1}$ in $\beta_{\tilde{I}}$. In general, suppose that $\bar{I} \in \widehat{B}_{p-1}$ precedes $\tilde{I}$ in $\beta$, where every $I_1 \in B_p$ in $\beta_{\bar{I}}$ is paired with an $I'_1 \in B_{p-1}-\widehat{B}_{p-1}$ in $\beta_{\bar{I}}$. Since $\beta_{\tilde{I}}$ has the same number of $p$-lanes and $(p-1)$-lanes, and since no elements of $\widehat{B}_{p-1}$ exist after $\bar{I}$ in $\beta_{\tilde{I}}$, $B_p$ has the same number of elements after $\bar{I}$ in $\beta_{\tilde{I}}$ as those of $B_{p-1}$ after $\bar{I}$ in $\beta_{\tilde{I}}$. Hence every $I_1 \in B_p$ after $\bar{I}$ in $\beta_{\tilde{I}}$ must be paired with an $I'_1 \in B_{p-1}-\widehat{B}_{p-1}$ after $\bar{I}$ in $\beta_{\tilde{I}}$. Therefore, inductively we conclude that every element of $B_p \cap C_{\mathrm{ht}(z_i)}$ must be paired with exactly one element of $B_{p-1}-\widehat{B}_{p-1}$. 
This shows that, after deleting all the elements of $C_{\mathrm{ht}(z_i)}$ from $B_p$, there are at least $|z_i|$ elements remaining in $B_p$ that can be lengthened, because $B_{p-1}-\widehat{B}_{p-1} \supset B_{p-1}-C_{\mathrm{ht}(z_i)}$. This shows that Procedure \ref{addnoncon} is executable.

\textit{$\beta^@$ satisfies the Lyndon property} \\
We first show that $\beta^@$ satisfies the Lyndon property for all letters $l \leq p-1$, 
so that we can apply Lemma \ref{lyndonimplies} to $\beta^@$ later on. Since $\beta^@$ clearly satisfies the Lyndon property for all letters $l \leq p-\mathrm{ht}(z_i)-1$, we only need to consider $l > p-\mathrm{ht}(z_i)-1$. Let $m$ be any positive integer. We analyze $\beta^@[m]$ using $\beta[m]$.

Suppose $p-\mathrm{ht}(z_i)-1 < l' \leq p-1$. By Lemma \ref{lyndon}, $\beta[m]$ has no fewer elements of $B_{l'}$ than elements of $B_{l'-1}$. 
\begin{clm}
\label{(claim1)}
$\beta[m]$ has no fewer elements of $B_{l'} - C_{\mathrm{ht}(z_i)}$ than elements of $B_{l'-1} - C_{\mathrm{ht}(z_i)}$.
\end{clm}
\begin{proof}
Suppose not. Then we must have $|B_{l'} \cap C_{\mathrm{ht}(z_i)}| > |B_{l'-1} \cap C_{\mathrm{ht}(z_i)}|$ in $\beta[m]$. Pick the minimal $m'$ such that $|B_{l'} \cap C_{\mathrm{ht}(z_i)}| = |B_{l'-1} \cap C_{\mathrm{ht}(z_i)}|$ in $\beta[m+m']$. By definition, any element of $B_{l'}$ in $\beta[m+m']$ outside $\beta[m]$ must be an element of $C_{\mathrm{ht}(z_i)}$ and must be paired with an element of $B_{l'-1} \cap C_{\mathrm{ht}(z_i)}$ to its right in $\beta[m+m']$, and $\beta[m+m']$ must end in an element of $B_{l'-1} \cap C_{\mathrm{ht}(z_i)}$. Notice that $|B_{l'} \cap C_{\mathrm{ht}(z_i)}| = |B_{l'-1} \cap C_{\mathrm{ht}(z_i)}|$ in $\beta[m+m']$ but $|B_{l'} - C_{\mathrm{ht}(z_i)}| < |B_{l'-1} - C_{\mathrm{ht}(z_i)}|$ in $\beta[m+m']$ due to the assumption that $|B_{l'} - C_{\mathrm{ht}(z_i)}| < |B_{l'-1} - C_{\mathrm{ht}(z_i)}|$ in $\beta[m]$. This implies that $\beta[m+m']$ has fewer elements of $B_{l'}$ than elements of $B_{l'-1}$, which contradicts the Lyndon property for $\beta$, and the claim is proved. 
\end{proof}

Suppose $l' = p-1$. By Item \ref{headrestr}, we have $\widehat{B}_{p-1} \subset C_{\mathrm{ht}(z_i)}$, so by definition only elements of $B_{p-1} - C_{\mathrm{ht}(z_i)} \subset B_{p-1} - \widehat{B}_{p-1}$ can be lengthened by Procedure \ref{addnoncon}. By Lemma \ref{snake} and Lemma \ref{no head}, any lower subinterval of $\beta$ whose entry $p$ has depth one must have either $p$ or $p-1$ as head. It follows that $\beta[j]$ has no fewer elements of $B_p$ than elements of $B_{p-1}-\widehat{B}_{p-1}$ for any $j$. 
\begin{clm}
\label{(claim2)}
$\beta[m]$ has no fewer elements of $B_p - C_{\mathrm{ht}(z_i)}$ than elements of $B_{p-1} - C_{\mathrm{ht}(z_i)}$.
\end{clm}
\begin{proof}
Suppose not. Then we must have $|B_p \cap C_{\mathrm{ht}(z_i)}| > |(B_{p-1}-\widehat{B}_{p-1}) \cap C_{\mathrm{ht}(z_i)}|$ in $\beta[m]$. Pick the minimal $b'$ such that $|B_p \cap C_{\mathrm{ht}(z_i)}| = |(B_{p-1}-\widehat{B}_{p-1}) \cap C_{\mathrm{ht}(z_i)}|$ in $\beta[m+b']$. By Item \ref{headrestr}, all of $\widehat{B}_{p-1}$ is contained inside $\beta[j_1]$, and $\beta[j_1]$ contains no elements of $B_{p-1} - C_{\mathrm{ht}(z_i)}$. Thus, we must have $j_1 < m$, and there are no elements of $\widehat{B}_{p-1}$ outside $\beta[m]$. By definition, any element of $B_p$ in $\beta[m+b']$ outside $\beta[m]$ must be an element of $C_{\mathrm{ht}(z_i)}$ and must be paired with an element of $B_{p-1} \cap C_{\mathrm{ht}(z_i)}$ to its right in $\beta[m+b']$. Notice that $|B_p \cap C_{\mathrm{ht}(z_i)}| = |(B_{p-1}-\widehat{B}_{p-1}) \cap C_{\mathrm{ht}(z_i)}|$ in $\beta[m+b']$ but $|B_{p} - C_{\mathrm{ht}(z_i)}| < |(B_{p-1}-\widehat{B}_{p-1}) - C_{\mathrm{ht}(z_i)}|$ in $\beta[m+b']$ due to the assumption that $|B_p - C_{\mathrm{ht}(z_i)}| < |B_{p-1} - C_{\mathrm{ht}(z_i)}|$ in $\beta[m]$. This implies that $|B_p| < |B_{p-1}-\widehat{B}_{p-1}|$ in $\beta[m+b']$, which is a contradiction, and the claim is proved.
\end{proof}

Now let $p-\mathrm{ht}(z_i)-1 < l'' \leq p-1$. By Lemma \ref{lyndon} and definition of pairing, $|B_{l''} \cap C_{\mathrm{ht}(z_i)}| \geq |B_{l''-1} \cap C_{\mathrm{ht}(z_i)}|$ in $\beta[m]$. By Claim \ref{(claim1)} for $l'' < p-1$ or Claim \ref{(claim2)} for $l'' = p-1$, we have $|B_{l''+1} - C_{\mathrm{ht}(z_i)}| \geq |B_{l''} - C_{\mathrm{ht}(z_i)}|$ and $|B_{l''} - C_{\mathrm{ht}(z_i)}| \geq |B_{l''-1} - C_{\mathrm{ht}(z_i)}|$ in $\beta[m]$. After running Procedure \ref{addnoncon}, we compare $|B_{l''}|$ in $\beta[m]$ with $|B_{l''}^@|$ in $\beta^@[m]$ and $|B_{l''-1}|$ in $\beta[m]$ with $|B_{l''-1}^@|$ in $\beta^@[m]$. If $|B_{l''-1} - C_{\mathrm{ht}(z_i)}| \geq n_i$ in $\beta[m]$, then Procedure \ref{addnoncon} forms $\beta^@[m]$ by lengthening $n_i$ elements of $B_{l''+1} - C_{\mathrm{ht}(z_i)}$, $n_i$ elements of $B_{l''} - C_{\mathrm{ht}(z_i)}$, and $n_i$ elements of $B_{l''-1} - C_{\mathrm{ht}(z_i)}$ in $\beta[m]$, and thus $B_{l''}$ in $\beta[m]$, $B_{l''}^@$ in $\beta^@[m]$ are equinumerous and $B_{l''-1}$ in $\beta[m]$ , $B_{l''-1}^@$ in $\beta^@[m]$ are equinumerous. If $|B_{l''-1} - C_{\mathrm{ht}(z_i)}| < n_i$ in $\beta[m]$, then Procedure \ref{addnoncon} forms $\beta^@[m]$ by lengthening $a_1$ elements of $B_{l''+1} - C_{\mathrm{ht}(z_i)}$, $a_2$ elements of $B_{l''} - C_{\mathrm{ht}(z_i)}$, and all elements of $B_{l''-1} - C_{\mathrm{ht}(z_i)}$ in $\beta[m]$, where $a_2 \leq a_1$. Therefore, in all cases we have $|B_{l''}^@| \geq |B_{l''-1}^@|$ in $\beta^@[m]$. Since $m$ was arbitrary, this completes the proof that $\beta^@$ satisfies the Lyndon property for all letters $l \leq p-1$.

\textit{Proof of Item \ref{addncworks}} \\
We show inductively that $l$-lanes of $\beta$ and $\beta^@$ are identical for all $l \geq p$, by comparing $\beta$ and $\beta^@$ one lower subinterval at a time, from left to right. In this case, given a lower subinterval or a portion of $\beta$, it will be obvious what we mean by the \textit{corresponding lower subinterval} or \textit{corresponding portion} of $\beta^{@}$, and vice versa.

Let $G_1$ denote the first lower subinterval of $\beta$ to be lengthened, and let $g'_1$ be the lane number of the head $p$ of $G_1$. By definition, the lanes of $\beta^@_{(G_1)_+}$ are identical to those of $\beta_{G_1}$. To determine the number of $(p-1)$-lanes in $\beta_{G_1}$, we need only determine the number of elements of $B_{p-1}$ in $\beta_{G_1}$, since $\beta^@_{(G_1)_+}$ and $\beta_{G_1}$ are identical, and since $\beta_{G_1}$ is a $(p-1)$-plateau. By Lemma \ref{snake}, any lower subinterval of $\beta^@_{(G_1)_+}$ whose entry $p$ has depth one must have either $p$ or $p-1$ as head, since $\beta^@_{(G_1)_+}$ is a $(p-1)$-plateau. Let $J$ denote the rightmost element of $\widehat{B}_{p-1}$ in $\beta_{G_1}$. The entries $p-1, p$ of $J$ have the same lane number, by Lemma \ref{snake}. All elements of $B_p$ in $\beta_{G_1}$ right of $J$ must be elements of $C_{\mathrm{ht}(z_i)}$ by definition, so they must 
be paired with the same number of elements of $B_{p-1} \cap C_{\mathrm{ht}(z_i)}$ in $\beta_{G_1}$ (which by definition do not contain the head of any $p$-lane). Since $\beta_{G_1}$ has the same number $g'_1-1$ of elements of $B_{p-1}$ as those of $B_p$, it follows that $(G_1)_+$ has entry $p-1$ with lane number $g'_1$, so its entry $p$ has lane number $g'_1$ as well.

Let $G_d$ be a lower subinterval of $\beta^@$ and let $G'_d$ denote the lower subinterval of $\beta$ corresponding to $G_d$. For our inductive hypothesis suppose that the $l$-lanes of $\beta^@_{G_d}$ are identical to those of $\beta_{G'_d}$ for all $l \geq p$. We show that the action of $G_d$ preserves this property; in fact, it suffices to show that entry $p$ has the same lane number in $G_d$ and $G'_d$. We treat separately the case that $G'_d$, $G_d$ are identical and the case that $G_d = (G'_d)_+$. Recall that both $\beta$ and $\beta^@$ contain a copy of $C_{\mathrm{ht}(z_i)}$, the context will make it clear which copy we refer to.

Consider the case that $G_d = (G'_d)_+$ and $\min G'_d = p$. Let $g^*_d$ be the lane number of the head $p$ of $G'_d$. Since $\beta^@$ satisfies the Lyndon property, $\beta^@_{G_d}$ is a $(p-1)^*$-plateau by Lemma \ref{lyndonimplies}, so the number of $(p-1)$-lanes in $\beta^@_{G_d}$ is the number of elements of $B_{p-1}$ in $\beta^@_{G_d}$. By Lemma \ref{snake}, any lower subinterval of $\beta^@_{G_d}$ whose entry $p$ has depth one must have either $p$ or $p-1$ as head, since $\beta^@_{G_d}$ is a $(p-1)^*$-plateau. Let $J_d$ denote the rightmost element of $\widehat{B}^@_{p-1}$ in $\beta^@_{G_d}$. The entries $p-1, p$ of $J_d$ have the same lane number, by Lemma \ref{snake}. All elements of $B^@_p$ in $\beta^@_{G_d}$ right of $J_d$ must be elements of $C_{\mathrm{ht}(z_i)}$ by definition (since they were fixed by Procedure \ref{addnoncon}), so they must be paired with the same number of elements of $B_{p-1} \cap C_{\mathrm{ht}(z_i)}$ in $\beta^@_{G_d}$ (which by definition do not contain the head of any $p$-lane). 
Since $\beta^@_{G_d}$ has the same number $g^*_d-1$ of $(p-1)$-lanes as $p$-lanes, it follows that $G_d$ has entry $p-1$ with lane number $g^*_d$, so its entry $p$ has lane number $g^*_d$ as well by the inductive hypothesis.

Consider the case that $G_d = (G'_d)_+$ and $\min G'_d \leq p-1$. Let $g_d$ be the lane number of the head of $G'_d$. Since $\beta^@$ satisfies the Lyndon property, $\beta^@_{G_d}$ is a $(p-1)^*$-plateau by Lemma \ref{lyndonimplies}, so the number of $(\min G'_d-1)$-lanes in $\beta^@_{G_d}$ is the number of elements of $B_{\min G_d}$ in $\beta^@_{G_d}$. Similarly, the number of $(\min G'_d)$-lanes in $\beta_{G'_d}$ is the number of elements of $B_{\min G'_d}$ in $\beta_{G'_d}$. The elements of $B_{\min G'_d} - C_{\mathrm{ht}(z_i)}$ in $\beta_{G'_d}$ are lengthened by Procedure \ref{addnoncon}, while the elements of $B_{\min G'_d} \cap C_{\mathrm{ht}(z_i)}$ in $\beta^@_{G_d}$ are paired with the same number of elements of $B_{\min G'_d -1} \cap C_{\mathrm{ht}(z_i)}$ in $\beta^@_{G_d}$ by definition. In addition, all elements of $B_{\min G'_d -1} - C_{\mathrm{ht}(z_i)}$ in $\beta_{G'_d}$ are lengthened by Procedure \ref{addnoncon} by Claim \ref{(claim1)}, since $|B_{\min G'_d} - C_{\mathrm{ht}(z_i)}| < n_i$ in $\beta_{G'_d}$.
Since $\beta^@_{G_d}$ has the same number $g_d-1$ of elements of $B^@_{\min G'_d -1}$ as elements of $B_{\min G'_d}$ in $\beta_{G'_d}$, it follows that $G_d$ has entry $\min G_d = \min G'_d -1$ with lane number $g_d$, so its entry $\min G'_d$ has lane number $g_d$ as well by the inductive hypothesis.

Consider the case that $G'_d$, $G_d$ are identical and $\min G_d < p-\mathrm{ht}(z_i)-1$. Since Procedure \ref{addnoncon} lengthens no lower subintervals with head smaller than $p-\mathrm{ht}(z_i)$, the $(\min G_d)$-lanes are identical in $\beta^@_{G_d}$ and $\beta_{G'_d}$. Thus, the entry $\min G_d$ has the same lane number $h_d$ in $G_d$ and $G'_d$. By Lemma \ref{lyndonimplies}, all entries of $G_d$ not exceeding $p-1$ have lane number $h_d$, and all entries of $G'_d$ not exceeding $p-1$ have lane number $h_d$. Since the $p$-lanes are identical in $\beta^@_{G_d}$ and $\beta_{G'_d}$ by the inductive hypothesis, entry $p$ has the same lane number in $G_d$ and $G'_d$.

Consider the case that $G'_d$, $G_d$ are identical and $\min G_d \geq p-\mathrm{ht}(z_i)-1$. By the inductive hypothesis, the $l$-lanes are identical in $\beta^@_{G_d}$ and $\beta_{G'_d}$ for all $l \geq p$, so we only need to consider the case $\min G_d \leq p-1$. We now apply Item \ref{headrestr}, with $a' = \mathrm{ht}(z_{i})$. 

If $\min G'_d = p-1$, then entry $p$ of $G'_d$ must have depth one or two, by Lemma \ref{snake}. Suppose that $\min G'_d = p-1$ and entry $p$ has depth one. Then $G'_d$ must lie inside $\beta[j_{p-\min G'_d}]$, and by Item \ref{headrestr} no element of $B_p \cup B_{p-1}$ in $\beta_{G'_d}$ can be lengthened by Procedure \ref{addnoncon}, so the entry $p-1$ has the same lane number in $G_d$ and $G'_d$, and hence the entry $p$ also has the same lane number in $G_d$ and $G'_d$ by the inductive hypothesis. 

Suppose that $\min G'_d = p-1$ and entry $p$ has depth two. Then $\beta_{G'_d}$ must have fewer $(p-1)$-lanes than $p$-lanes. By Claim \ref{(claim2)}, $|B_p-C_{\mathrm{ht}(z_i)}| \geq |B_{p-1}-C_{\mathrm{ht}(z_i)}|$ in $\beta_{G'_d}$. If $|B_{p-1}-C_{\mathrm{ht}(z_i)}| \geq n_i$ in $\beta_{G'_d}$, then Procedure \ref{addnoncon} lengthens $n_i$ elements of $B_p$ and $n_i$ elements of $B_{p-1}$ in $\beta_{G'_d}$, so the number of $(p-1)$-lanes in $\beta^@_{G_d}$ equals that of $(p-1)$-lanes in $\beta_{G'_d}$ and remains less than that of $p$-lanes in $\beta^@_{G_d}$, and hence entry $p$ of $G_d$ has depth two. In the case $|B_{p-1}-C_{\mathrm{ht}(z_i)}| < n_i$ in $\beta_{G'_d}$, we must have $G'_d \in C_{\mathrm{ht}(z_i)}$. If $G'_d \in \widehat{B}_{p-1}$, then no element of $B_p$ in $\beta_{G'_d}$ is lengthened, by Item \ref{headrestr}, so the number of $(p-1)$-lanes does not increase. Suppose $G'_d \in (B_{p-1} \cap C_{\mathrm{ht}(z_i)})-\widehat{B}_{p-1}$. Since each $I \in \widehat{B}_{p-1}$ in $\beta_{G'_d}$ contributes a new $(p-1)$-lane and a new $p$-lane, we can exclude $\widehat{B}_{p-1}$ from consideration. Procedure \ref{addnoncon} fixes the elements of $C_{\mathrm{ht}(z_i)}$, and lengthens all elements of $B_{p-1}-C_{\mathrm{ht}(z_i)}$ in $\beta_{G'_d}$ and some elements of $B_p-C_{\mathrm{ht}(z_i)}$ in $\beta_{G'_d}$, and we have $|(B_{p-1} \cap C_{\mathrm{ht}(z_i)})-\widehat{B}_{p-1}| < |B_{p} \cap C_{\mathrm{ht}(z_i)}|$ in $\beta_{G'_d}$ since $G'_d \in C_{\mathrm{ht}(z_i)}$ lies outside $\beta_{G'_d}$. 
It follows that the number of $(p-1)$-lanes in $\beta^@_{G_d}$ is again less than that of $p$-lanes in $\beta^@_{G_d}$, and hence entry $p$ of $G_d$ has depth two. Thus, in both cases entry $p$ of $G_d$ has depth two, so it has the same lane number in $G_d$ and $G'_d$ by the inductive hypothesis.

Suppose $\min G'_d < p-1$ and all elements of $B_{\min G'_d +1}$ in $\beta_{G'_d}$ belong to $C_{\mathrm{ht}(z_i)}$. Then all elements of $B_{\min G'_d}$ in $\beta_{G'_d}$ must belong to $C_{\mathrm{ht}(z_i)}$ as well, by Claim \ref{(claim1)}. It follows that no element of $B_{\min G'_d +1}$ in $\beta_{G'_d}$ can be lengthened by Procedure \ref{addnoncon}, and no element of $B_{\min G'_d}$ in $\beta_{G'_d}$ can be lengthened by Procedure \ref{addnoncon}, so the number of elements of $B_{\min G'_d}$ is the same in $\beta^@_{G_d}$ and $\beta_{G'_d}$. Thus, the entries $\min G'_d \leq l \leq p-1$ have the same lane number $k_d$ in $G_d$ and $G'_d$ by Lemma \ref{lyndonimplies}, so the entry $p$ also has the same lane number in $G_d$ and $G'_d$ by the inductive hypothesis. 

Suppose $\min G'_d < p-1$ and $\beta_{G'_d}$ contains some element $H \in B_{\min G'_d +1}-C_{\mathrm{ht}(z_i)}$. Let $H^*$ be the element of $B_{\min G'_d +1}$ left of $G'_d$. By Item \ref{headrestr}, $H$ must lie outside $\beta[j_{p-(\min G'_d +1)}] = \beta[j_{p-\min G'_d -1}]$, so $H^*$ must lie outside $\beta[j_{p-\min G'_d -1}]$ as well. Applying the second part of Item \ref{headrestr}, we know that all entries of $H^*$ not exceeding $p$ have the same lane number $d^*$, $H^*$ has entry $p$ of depth $p-(\min G'_d +1)+1 = p-\min G'_d$, and $H^*$ has entry $p-1$ of depth $p-\min G'_d-1$. By Lemma \ref{lyndonimplies}, all entries of $G'_d$ not exceeding $p-1$ have a common lane number $l_1$, and all entries of $G_d$ not exceeding $p-1$ have a common lane number $l_2$. By Lemma \ref{lyndon}, we have $l_1 \leq d^*$, so $l_1 < d^*+1$. By Lemma \ref{snake}, in $G'_d$ the depth of entry $p$ exceeds that of entry $p-1$ by at most one, meaning that entry $p$ of $G'_d$ has depth at most $(p-\min G'_d) +1$ (since $G'_d$ has entry $p-1$ of depth $(p-\min G'_d-1) +1 = p-\min G'_d$ by Lemma \ref{lyndonimplies}). By Claim \ref{(claim1)}, Procedure \ref{addnoncon} lengthens no fewer elements of $B_{\min G'_d +1}$ than elements of $B_{\min G'_d}$ in $\beta_{G'_d}$. Looking at $\beta^@_{G_d}$, this means that $l_2 \geq l_1$, and hence the entry $p$ of $G_d$ has depth not exceeding that of the entry $p$ of $G'_d$, by the inductive hypothesis. On the other hand, we now show that entry $p$ of $G_d$ has depth at least $p-\min G'_d +1$. If $|B_{\min G'_d}-C_{\mathrm{ht}(z_i)}| \geq n_i$ in $\beta_{G'_d}$, then Procedure \ref{addnoncon} lengthens $n_i$ elements of $B_{\min G'_d +1}$ and $n_i$ elements of $B_{\min G'_d}$ in $\beta_{G'_d}$, so we must have $l_2 = l_1 \leq d^*$. If $|B_{\min G'_d}-C_{\mathrm{ht}(z_i)}| < n_i$ in $\beta_{G'_d}$, then $G'_d \in C_{\mathrm{ht}(z_i)}$, $|(B_{\min G'_d} \cap C_{\mathrm{ht}(z_i)})| < |B_{\min G'_d +1} \cap C_{\mathrm{ht}(z_i)}|$ in $\beta_{G'_d}$ since $G'_d$ lies outside $\beta_{G'_d}$, and Procedure \ref{addnoncon} lengthens all elements of $B_{\min G'_d}-C_{\mathrm{ht}(z_i)}$ in $\beta_{G'_d}$ and some elements of $B_{\min G'_d +1}-C_{\mathrm{ht}(z_i)}$ in $\beta_{G'_d}$, so we must have $l_2-1 < d^*$ as well. Thus, in both cases the entry $p$ of $G_d$ has depth exceeding $p-\min G'_d$. This shows that entry $p$ has depth exactly $p-\min G'_d +1$ in both $G'_d$ and $G_d$, so its lane number is the same in both $G'_d$ and $G_d$ by the inductive hypothesis. This completes the proof that the $l$-lanes are identical in $\beta$ and $\beta^@$ for all $l \geq p$.

Since $\beta^@$ satisfies the Lyndon property, $\beta^@$ is a $(p-1)^*$-plateau. We now show that $\beta^@$ is indeed obtained from $\beta$ by lengthening $n_i$ $(p-1)$-lanes of length $\mathrm{ht}(z_i)$. Any initial segment $\beta^@[j]$ contains no fewer lower subintervals with head $l$ than those with head $l-1$, for all $l \leq p-1$. Notice that, compared to $\beta$, $\beta^@$ has the same number of elements of $B_m$ for all $m > p-\mathrm{ht}(z_i)$ and $m < p-\mathrm{ht}(z_i)-1$, but has $n_i$ more elements of $B_{p-\mathrm{ht}(z_i)-1}$ and $n_i$ fewer elements of $B_{p-\mathrm{ht}(z_i)}$. By Lemma \ref{lyndonimplies}, this shows that $\beta^@$ corresponds to the rigged configuration obtained after adding $n_i$ noncontributing boxes beneath $z_i$. It follows that in particular $\beta^@$ is a $(p,\mathrm{ht}(z_{i}),r)$-plateau.

\textit{Proof of Item \ref{headrestr}} \\
Finally, we prove that Item \ref{headrestr} holds for $\beta^@$ for any integer $0 < s' < \min(\mathrm{ht}(z_i),r)$. Since $\beta$ is a $(p-1)^*$-plateau, any lower subinterval containing the head of some $p$-lane must have head $p$ or $p-1$. By Claim \ref{(claim2)}, if $j_1^@ \geq j_1$ is minimal such that $\beta^@[j_1^@]$ contains all $n_i$ added copies of $p-2$, then $\beta^@[j_1^@]$ contains all the lower subintervals with head $p-1$ containing the head of some $p$-lane, and $\beta^@[j_1^@]$ contains no more lower subintervals with head $p-1$ than lower subintervals with head $p-2$; by definition, any $I_1 \in B_{p-1} \cap C_{\mathrm{ht}(z_i)}$ inside $\beta[j_1^@]$ must be paired with an $I_2 \in B_{p-2} \cap C_{\mathrm{ht}(z_i)}$ after $I_1$ in $\beta[j_1^@]$, so that $\beta^@[j_1^@]$ contains both $I_1$, $I_2$. By Claim \ref{(claim1)}, for all $1 < d \leq s'$, if $j_d^@ \geq \max(j_d,j_{d-1}^@)$ is minimal such that $\beta^@[j_d^@]$ contains all $n_i$ added copies of $p-d-1$, then $\beta^@[j_d^@]$ contains no more lower subintervals with head $p-d$ than lower subintervals with head $p-d-1$; by definition, any $J_1 \in B_{p-d} \cap C_{\mathrm{ht}(z_i)}$ inside $\beta[j_d^@]$ must be paired with a $J_2 \in B_{p-d-1} \cap C_{\mathrm{ht}(z_i)}$ after $J_1$ in $\beta[j_d^@]$, so that $\beta^@[j_d^@]$ contains both $J_1$, $J_2$.

Let $I'$ be an element of $B^@_{p-h}$ outside $\beta^@[j_{h}^@]$, where $1 \leq h \leq s'$. Denote by $I^*$ the lower subinterval of $\beta$ corresponding to $I'$. By definition, Procedure \ref{addnoncon} must have added all $n_i$ copies of $p-h$ and all $n_i$ copies of $p-h-1$ to $\beta[j_{h}^@]$, so $I'$ must be identical to $I^*$ (since $I'$ was not obtained by lengthening $I^*$). Since $I^*$ must lie outside $\beta[j_{h}]$, $I^*$ has entry $p$ of depth $p-(p-h)+1 = h+1$, entry $p-1$ of depth $p-1-(p-h)+1 = h$, and one common lane number $i^*$ for all entries not exceeding $p$, by Item \ref{headrestr} for $\beta$. As already shown, $\beta^@$ is a $(p-1)^*$-plateau with identical $l$-lanes to those of $\beta$, for all $l \geq p$. Since $\beta^@[j_{h}^@]$ contains the same number of newly added copies of $p-h$ as newly added copies of $p-h-1$, the elements of $B^@_{p-h}$ in $\beta^@[j_{h}^@]$ must be equinumerous with the elements of $B_{p-h}$ in $\beta[j_{h}^@]$, so $\min I'$ has the same lane number $i^*$ as $\min I^*$. By Lemma \ref{lyndonimplies} for $\beta^@$, all entries of $I'$ not exceeding $p-1$ have lane number $i^*$ and entry $p-1$ has depth $h$. Entry $p$ of $I'$ also has lane number $i^*$ and has depth $h+1$ because $\beta^@_{I'}$ and $\beta_{I^*}$ have identical $p$-lanes.

Lastly, let $\widetilde{I}$ denote the rightmost element of $\widehat{B}_{p-1}$ in $\beta$. Let $C_{a''}^@$ denote the set of deleted elements of $\beta^@$, in the context of applying Procedure \ref{addnoncon} to $\beta^@$, where the smallest entry to be added to $\beta^@$ is $p-a''-1 \geq p-s'-1$. By Item \ref{headrestr} for $\beta$, no element of $B_p$ left of $\widetilde{I}$ can be lengthened by Procedure \ref{addnoncon}. It follows that Procedure \ref{addnoncon} must lengthen the first $n_i$ elements of $B_p - C_{\mathrm{ht}(z_i)}$ after $\widetilde{I}$. Therefore, if $I'' \in B^@_{p-1}$ contains the head of some $p$-lane, then all elements of $B^@_p$ in $\beta^@_{I''}$ must be elements of $C_{a''}^@$; we have $B_{p} \cap C_{\mathrm{ht}(z_i)} \subset B^@_{p} \cap C_{a''}^@$, because $B^@_{p-a''-1} \subset C_{a''}^@$, $B_{p-a''-1} \cap C_{\mathrm{ht}(z_i)} \subset B^@_{p-a''-1}$, $B_{b''} \cap C_{\mathrm{ht}(z_i)} \subset B^@_{b''} \cap C_{a''}^@$ for all $p-a''-1 < b'' \leq p-1$, and $B^@_{p} \subset B_p$. This completes the proof of Item \ref{headrestr}.
\end{proof}

\subsection{Growth Algorithm}
\label{growth}
Let $R' = (\mu_1,\mu_2,\ldots,\mu_n)$ be a ${\cal B}(\infty)$ rigged configuration of type $A_n$. For any $d \in [n]$, label from left to right the stretches of $\mu_d$ by $\tilde{\mu}_{d}^{b_1}, \tilde{\mu}_{d}^{b_2}, \ldots, \tilde{\mu}_{d}^{b_k}$. The following lemma determines how the riggings of $\mu_d$ constrain $\mu_{d-1}$. If $z$ is a stretch of $\mu_j$ intersecting the $l$th column $y_l$ of $\mu_j$, we say that $z$ \textbf{spans} $y_l$. If $y'_l$ is the $l$th column of any other partition, we also say that $z$ \textbf{spans} $y'_l$.
\begin{lem}
\label{rigging desc}
The rigging $r_d^{b_l}$ of the row $\mu_d^{b_l}$ containing $\tilde{\mu}_d^{b_l}$ can be written as \[r_d^{b_l} = \sum_{i=1}^{l}{-\mathrm{cb}(\mu_{d})^{b_i} + \mathrm{acon}_{d}^{b_i}},\]  where $\mathrm{cb}(\mu_{d})^{b_i}$ is the number of contributing boxes in $\tilde{\mu}_d^{b_i}$, and $0 \leq \mathrm{acon}_{d}^{b_i} \leq |\tilde{\mu}_d^{b_i}|$ is the \textbf{above contribution to} $\tilde{\mu}_d^{b_i}$ from $\mu_{d-1}$; to be precise, the above contribution to $\tilde{\mu}_d^{b_i}$ is the number of columns of $\mu_{d-1}$ spanned by the stretch $\tilde{\mu}_d^{b_i}$ that end in a contributing box.
\end{lem}
\begin{proof}
Follows from Lemma \ref{lane rig}. To be precise, the rigging of the row $\mu_d^{b_l}$ is determined by the number of contributing boxes added to the first $|\mu_d^{b_l}|$ columns of the $d$th partition and the number of contributing boxes that will be added to the corresponding columns of the $(d-1)$st partition; the former number corresponds to the sum $\sum_{i=1}^{l}{-\mathrm{cb}(\mu_{d})^{b_i}}$, while the latter number corresponds to the sum $\sum_{i=1}^{l}{\mathrm{acon}_{d}^{b_i}}$.
\end{proof}
\begin{rem}
Notice that fixing the riggings of $\mu_d$ is the same as fixing all $\mathrm{acon}_{d}^{b_i}$. Once we fix $\mathrm{acon}_{d}^{b_i}$ for all $i \leq l$, then the number of columns of $\mu_{d-1}$ spanned by $\tilde{\mu}_d^{b_i}$ ending in a contributing box is completely determined (and must equal $\mathrm{acon}_{d}^{b_i}$), for all $i \leq l$. In what follows, we will always write the rigging of a given row in the form shown in Lemma \ref{rigging desc}.
\end{rem}

Here we give an algorithm for growing all $A_n$ rigged configurations, which can be used to check inductively (starting with the last partition and going backward) whether a given tuple of rigged partitions is a legitimate $A_n$ rigged configuration. Given $\mu_d$ and all $\mathrm{acon}_{d}^{b_i}$ fixed (equivalently, all riggings fixed), this algorithm produces $\mu_{d-1}$ by utilizing Procedures \ref{addnoncon} and \ref{addcon} to add noncontributing boxes and contributing boxes beneath the stretches of $\overline{\mu_d}$, one stretch at a time, from left to right. We will always follow Convention \ref{invisible} on labeling the stretches of $\mu_d$ and the stretches of $\overline{\mu_d}$. By combining Theorem \ref{constraints}, Lemma \ref{max rows}, Lemma \ref{rigging desc}, and Lemma \ref{addboxproof}, we have the following \textbf{growth algorithm} characterizing all $A_n$ rigged configurations:

\begin{thm}
\label{growth alg thm}
Let $\Lambda = (\lambda_1,\lambda_2,\ldots,\lambda_n)$ be a tuple of rigged partitions. Then $\Lambda$ is a ${\cal B}(\infty)$ rigged configuration of $A$-type if and only if $\Lambda$ satisfies the following:
\begin{enumerate}
\item $\lambda_{n}$ must consist of a single row with rigging $r_{n} = -|\lambda_{n}| + \mathrm{acon}_{n}$, where $\mathrm{acon}_{n}$ is an integer $0 \leq \mathrm{acon}_{n} \leq |\lambda_{n}|$.
\item In general, $\lambda_{n-i-1}$ is determined by $\lambda_n, \lambda_{n-1}, \ldots, \lambda_{n-i}$ as follows. Label the stretches of $\lambda_{n-i}$ by $\tilde{\lambda}_{n-i}^{b_1}, \tilde{\lambda}_{n-i}^{b_2}, \ldots, \tilde{\lambda}_{n-i}^{b_k}$, with corresponding rows $\lambda_{n-i}^{b_1},$ $\lambda_{n-i}^{b_2},$ $\ldots,$ $\lambda_{n-i}^{b_k}$ respectively. Write the rigging of row $\lambda_{n-i}^{b_j}$ as $r_{n-i}^{b_j} = \sum_{m=1}^{j}{-\mathrm{cb}(\lambda_{n-i})^{b_m} + \mathrm{acon}_{n-i}^{b_m}}$.

To begin, we have $\lambda_{n-i-1} \supset \overline{\lambda_{n-i}}$. Label the stretches of $\overline{\lambda_{n-i}}$ by $s^{b_1}, s^{b_2}, \ldots, s^{b_{k}}$, where $s^{b_j}$ is simply a copy of $\tilde{\lambda}_{n-i}^{b_j}$. We can now describe $\lambda_{n-i-1}$ by specifying how many boxes $\lambda_{n-i-1}$ can have beneath each stretch $s^{b_j}$ and what the riggings are.

Fix $j \in [k]$, let $\eta^j_1$ and $\eta^j_2$ denote the first row and second row beneath $s^{b_j}$ in $\lambda_{n-i-1}$, respectively. We have $|\eta^j_1| = \mathrm{ncb}(\lambda_{n-i-1})^{b_j} + \mathrm{cb}(\lambda_{n-i-1})_1^{b_j}$ and $|\eta^j_2| = \mathrm{cb}(\lambda_{n-i-1})_2^{b_j}$, where $0 \leq \mathrm{ncb}(\lambda_{n-i-1})^{b_j} \leq |s^{b_j}|$, $0 \leq \mathrm{cb}(\lambda_{n-i-1})_2^{b_j} \leq \mathrm{ncb}(\lambda_{n-i-1})^{b_j}$, $0 \leq \mathrm{cb}(\lambda_{n-i-1})_1^{b_j} \leq |s^{b_j}| - \mathrm{ncb}(\lambda_{n-i-1})^{b_j}$, and $\mathrm{cb}(\lambda_{n-i-1})_2^{b_j} + \mathrm{cb}(\lambda_{n-i-1})_1^{b_j} = \mathrm{acon}_{n-i}^{b_j}$. Let $Y_j := \sum_{u=1}^{j-1}{|s^{b_u}|}$. 

\begin{enumerate}
\item Let $\delta(\lambda_{n-i-1}) := \mathrm{maxr}_{n-i-1} - \max((\overline{\lambda_{n-i}})^t)$. At most $\min(\delta(\lambda_{n-i-1}),2)$ rows can exist beneath $s^{b_1}$ in $\lambda_{n-i-1}$. Any row of length $|\eta^1_2|$ has rigging $r_{n-i-1}^{b_1,2} := -\mathrm{cb}(\lambda_{n-i-1})_2^{b_1}$ $+ \mathrm{acon}_{n-i-1}^{b_1,2}$, any row of length $|\eta^1_1|$ has rigging $r_{n-i-1}^{b_1,1} := r_{n-i-1}^{b_1,2} - \mathrm{cb}(\lambda_{n-i-1})_1^{b_1} + \mathrm{acon}_{n-i-1}^{b_1,1}$, and any row of length $|s^{b_1}|$ has rigging $r_{n-i-1}^{b_1,*} = r_{n-i-1}^{b_1,1} + \mathrm{acon}_{n-i-1}^{b_1,*}$, where $0 \leq \mathrm{acon}_{n-i-1}^{b_1,2} \leq \mathrm{cb}(\lambda_{n-i-1})_2^{b_1}$, $0 \leq \mathrm{acon}_{n-i-1}^{b_1,1} \leq |\eta^1_1| - \mathrm{cb}(\lambda_{n-i-1})_2^{b_1}$, and $0 \leq \mathrm{acon}_{n-i-1}^{b_1,*} \leq \Upsilon(b_1)$ where
$\Upsilon(b_1) =
\left\{
	\begin{array}{ll}
		0  & \mbox{if } \delta(\lambda_{n-i-1})=0 \\
		|s^{b_1}| - |\eta^1_1| & \mbox{otherwise}
	\end{array}
\right.$

\item For any $2 \leq m \leq k$, we now determine the rigging for the stretch $s^{b_m}$, given that we have already done so for $s^{b_1}, s^{b_2}, \ldots, s^{b_{m-1}}$. Let $N(b_{m})$ be the height of the column of $\lambda_{n-i-1}$ on the left of $s^{b_m}$. At most $\min(N(b_{m})-\mathrm{ht}(s^{b_m}),2)$ rows can exist beneath $s^{b_m}$ in $\lambda_{n-i-1}$. Let $U_{b_m}$ denote the uppermost row of $\lambda_{n-i-1}$ with length at most $Y_m$, and let $r(b_m)$ denote the rigging of $U_{b_m}$. There are three cases:
\begin{enumerate}
\item If $|U_{b_m}| < Y_m$ and $|\eta^m_2| \neq 0$, then any row of length $Y_m + |\eta^m_2|$ has rigging $r_{n-i-1}^{b_m,2} = r(b_m) - \mathrm{cb}(\lambda_{n-i-1})_1^{b_{m-1}} - \mathrm{cb}(\lambda_{n-i-1})_2^{b_m} + \mathrm{acon}_{n-i-1}^{b_m,2}$ and any row of length $Y_m + |\eta^m_1|$ has rigging $r_{n-i-1}^{b_m,1} = r_{n-i-1}^{b_m,2} - \mathrm{cb}(\lambda_{n-i-1})_1^{b_m} + \mathrm{acon}_{n-i-1}^{b_m,1}$, where $0 \leq \mathrm{acon}_{n-i-1}^{b_m,2} \leq Y_m + |\eta^m_2| - |U_{b_m}|$ and $0 \leq \mathrm{acon}_{n-i-1}^{b_m,1} \leq |\eta^m_1|-|\eta^m_2|$.
\item If $|U_{b_m}| < Y_m$ and $|\eta^m_2| = 0$, then any row of length $Y_m + |\eta^m_1|$ has rigging $r_{n-i-1}^{b_m,1} = r(b_m) - \mathrm{cb}(\lambda_{n-i-1})_1^{b_m} + \mathrm{acon}_{n-i-1}^{b_m,1}$, where $0 \leq \mathrm{acon}_{n-i-1}^{b_m,1} \leq Y_m + |\eta^m_1| - |U_{b_m}|$.
\item Otherwise, any row of length $Y_m + |\eta^m_2|$ has rigging $r_{n-i-1}^{b_m,2} = r(b_m) - \mathrm{cb}(\lambda_{n-i-1})_2^{b_m}$ $+ \mathrm{acon}_{n-i-1}^{b_m,2}$ and any row of length $Y_m + |\eta^m_1|$ has rigging $r_{n-i-1}^{b_m,1} = r_{n-i-1}^{b_m,2} - \mathrm{cb}(\lambda_{n-i-1})_1^{b_m} + \mathrm{acon}_{n-i-1}^{b_m,1}$, where $0 \leq \mathrm{acon}_{n-i-1}^{b_m,2} \leq \mathrm{cb}(\lambda_{n-i-1})_2^{b_m}$ and $0 \leq \mathrm{acon}_{n-i-1}^{b_m,1} \leq |\eta^m_1|-|\eta^m_2|$.
\end{enumerate}
In all cases, any row of length $Y_m + |s^{b_m}|$ has rigging $r_{n-i-1}^{b_m,*} = r_{n-i-1}^{b_m,1} + \mathrm{acon}_{n-i-1}^{b_m,*}$, where $0 \leq \mathrm{acon}_{n-i-1}^{b_m,*} \leq |s^{b_m}|-|\eta^m_1|$.
\item Finally, to determine the first row $\lambda_{n-i-1}^1$ of $\lambda_{n-i-1}$, we compare $|\eta^k_1|$ with $|\tilde{\lambda}_{n-i}^1|$.
\begin{enumerate}
\item If $|\eta^k_1| < |\tilde{\lambda}_{n-i}^1|$ and $|\lambda_{n-i}^1| > |\lambda_{n-i}^2|$, then $|\lambda_{n-i-1}^1| = |\lambda_{n-i}^2| + |\eta^k_1|$, with rigging $r_{n-i-1}^{b_k,1}$.
\item Otherwise, we have $|\lambda_{n-i-1}^1| = |\lambda_{n-i}^1| + \mathrm{cb}'(\lambda_{n-i-1})^{b_k}$, where $\mathrm{cb}'(\lambda_{n-i-1})^{b_k}$ can be any nonnegative integer. Let $r(\lambda_{n-i-1}^2)$ denote the rigging of $\lambda_{n-i-1}^2$. 
\begin{enumerate}
\item If $|\lambda_{n-i}^1| = |\lambda_{n-i}^2|$, then $\lambda_{n-i-1}^1$ has rigging $r_{n-i-1}^{b_k,!} = r(\lambda_{n-i-1}^2) - \mathrm{cb}'(\lambda_{n-i-1})^{b_k} + \mathrm{acon}_{n-i-1}^{b_k,!}$, where $0 \leq \mathrm{acon}_{n-i-1}^{b_k,!} \leq |\lambda_{n-i-1}^1| - |\lambda_{n-i-1}^2|$.
\item If $|\lambda_{n-i}^1| > |\lambda_{n-i}^2|$, then $\lambda_{n-i-1}^1$ has rigging $r_{n-i-1}^{b_k,!} = r(\lambda_{n-i-1}^2) - \mathrm{cb}(\lambda_{n-i-1})_1^{b_k} - \mathrm{cb}'(\lambda_{n-i-1})^{b_k} + \mathrm{acon}_{n-i-1}^{b_k,!}$, where $0 \leq \mathrm{acon}_{n-i-1}^{b_k,!} \leq |\lambda_{n-i-1}^1| - |\lambda_{n-i-1}^2|$.
\end{enumerate}
\end{enumerate}
\end{enumerate}
\end{enumerate}
\end{thm}
\begin{rem}
Here, $\mathrm{ncb}(\lambda_{n-i-1})^{b_j}$ denotes the number of noncontributing boxes in the first row beneath $s^{b_j}$, $\mathrm{cb}(\lambda_{n-i-1})_1^{b_j}$ denotes the number of contributing boxes in the first row beneath $s^{b_j}$, and $\mathrm{cb}(\lambda_{n-i-1})_2^{b_j}$ denotes the number of contributing boxes in the second row beneath $s^{b_j}$.
\end{rem}
\begin{proof}
Observe that, in short, the theorem states that the full range of boxes allowed under Theorem \ref{constraints} can indeed be added to $\overline{\lambda_{n-i}}$; in other words, the constraints imposed by Theorem \ref{constraints} are tight. Starting with the empty rigged configuration $R_{\empty}$, the construction of $\Lambda$ using Procedures \ref{addnoncon} and \ref{addcon} is described item by item as follows.
\begin{enumerate}
\item By Lemma \ref{addboxproof}, $\lambda_{n}$ can be formed by using Procedure \ref{addcon} to add $|\lambda_{n}|$ contributing boxes to the empty $n$th partition, and the resulting $(n-1)$st partition is $\overline{\lambda_{n}} = \emptyset$.
\item The base case of $\lambda_{n}$ justifies the inductive hypothesis that the $(n-i-1)$st partition is $\overline{\lambda_{n-i}}$ (with riggings of zero) after the partitions $\lambda_n, \lambda_{n-1}, \ldots, \lambda_{n-i}$ have been constructed, in that order, by Procedures \ref{addnoncon} and \ref{addcon}. Lemma \ref{addboxproof} ensures that Procedures \ref{addnoncon} and \ref{addcon} can add all the boxes Theorem \ref{constraints} allows under $s^{b_j}$. To be precise, Procedure \ref{addnoncon} will be used to add the $\mathrm{ncb}(\lambda_{n-i-1})^{b_j}$ noncontributing boxes to the first row beneath $s^{b_j}$, Procedure \ref{addcon} will be used to add the $\mathrm{cb}(\lambda_{n-i-1})_1^{b_j}$ contributing boxes after these $\mathrm{ncb}(\lambda_{n-i-1})^{b_j}$ noncontributing boxes in the same row, and Procedure \ref{addcon} will be used to add the $\mathrm{cb}(\lambda_{n-i-1})_2^{b_j}$ contributing boxes beneath these $\mathrm{ncb}(\lambda_{n-i-1})^{b_j}$ noncontributing boxes. Finally, by Lemma \ref{rigging desc}, a total of $\mathrm{acon}_{n-i}^{b_j}$ contributing boxes must be added beneath $s^{b_j}$ to account for the positive contribution to $\tilde{\lambda}_{n-i}^{b_j}$. In the cases (a) and (b), Theorem \ref{constraints} determines how many rows can be added beneath $s^{b_j}$ and how many contributing boxes can be added beneath the corresponding stretch of the $(n-i-2)$nd partition.
\begin{enumerate}
\item Notice that, for either of the Procedures \ref{addnoncon} and \ref{addcon} to work, we must have $p-\mathrm{ht}(z_1)-1 \geq 1$ or equivalently $(p-1) - \mathrm{ht}(z_1) > 0$. Indeed, by Lemma \ref{lyndonimplies}, $\min((p-1) - \mathrm{ht}(z_1),2)$ is the number of rows that can be added beneath the first stretch by Procedures \ref{addnoncon} and \ref{addcon}. 

By Theorem \ref{constraints}, at most $\min(\delta(\lambda_{n-i-1}),2)$ rows can exist beneath $s^{b_1}$ in $\lambda_{n-i-1}$. To prove the converse, let $H_{l} \leq \mathrm{maxr}_{l}$ denote the number of rows in $\lambda_{l}$. Then $\overline{\lambda_{n-i}}$ has $\overline{H_{n-i-1}} = H_{n-i}-1 \leq \mathrm{maxr}_{n-i}-1$ rows. By above, $\min(n-i-1 - \overline{H_{n-i-1}},2) = \min(n-i - H_{n-i},2)$ is exactly the number of rows allowed to be added beneath $s^{b_1}$ by Procedures \ref{addnoncon} and \ref{addcon}. Consider the cases $n-i-1 > \frac{n+1}{2}$ and $n-i-1 \leq \frac{n+1}{2}$. If $n-i-1 \leq \frac{n+1}{2}$, then $\delta(\lambda_{n-i-1}) = n-i-1 - \overline{H_{n-i-1}}$ since $\mathrm{maxr}_{n-i-1} = n-i-1$. Suppose $n-i-1 > \frac{n+1}{2}$. Then $n-i-1 - \overline{H_{n-i-1}} \geq n-i-1 - (\mathrm{maxr}_{n-i}-1) = n-i-1 - (n-(n-i)+1-1) = n-i-1 -i = n-2i-1 \geq 2$ by Lemma \ref{max cases}, since $i < \frac{n-3}{2}$. Also, $\delta(\lambda_{n-i-1}) = \mathrm{maxr}_{n-i-1} - \overline{H_{n-i-1}} = n-(n-i-1)+1 - \overline{H_{n-i-1}} = i+3 - H_{n-i} \geq i+3 - \mathrm{maxr}_{n-i} = i+3 - (i+1) = 2$ by Lemma \ref{max cases}. It follows that $\min(\delta(\lambda_{n-i-1}),2) = 2 = \min(n-i - H_{n-i},2)$ as well. Thus, in both cases $\min(\delta(\lambda_{n-i-1}),2)$ rows can indeed be added beneath $s^{b_1}$.

Now we prove the riggings formulas. In the case $\delta(\lambda_{n-i-1})=0$, both $\eta^j_1$ and $\eta^j_2$ must be empty. Since $\delta(\lambda_{n-i-1})=0$, no rows can be added beneath the first stretch of the $(n-i-2)$nd partition, which already has $\mathrm{maxr}_{n-i-2}$ rows by Lemma \ref{max rows}, and hence we must have $\mathrm{acon}_{n-i-1}^{b_1,*} = 0$. The claims for the other case is obvious.
\item By Theorem \ref{constraints}, at most $\min(N(b_{m})-\mathrm{ht}(s^{b_m}),2)$ rows can exist beneath $s^{b_m}$ in $\lambda_{n-i-1}$; if more than $N(b_{m})-\mathrm{ht}(s^{b_m})$ rows is added beneath $s^{b_m}$, the result will not be a valid partition. The converse, that $\min(N(b_{m})-\mathrm{ht}(s^{b_m}),2)$ rows can indeed be added beneath $s^{b_m}$, follows from Lemma \ref{addboxproof}.
\begin{enumerate}
\item Suppose $|U_{b_m}| < Y_m$ and $|\eta^m_2| \neq 0$, let $h_1$ denote the bottommost row of length $Y_m + |\eta^m_2|$, and let $h_2$ denote the stretch of $\lambda_{n-i-1}$ contained in $h_1$. Since $h_2$ contains $\mathrm{cb}(\lambda_{n-i-1})_1^{b_{m-1}} + \mathrm{cb}(\lambda_{n-i-1})_2^{b_m}$ contributing boxes and $|h_2| = Y_m + |\eta^m_2| - |U_{b_m}|$, we obtain the formula for $r_{n-i-1}^{b_m,2}$.
\item Suppose $|U_{b_m}| < Y_m$ and $|\eta^m_2| = 0$. Then $\lambda_{n-i-1}$ has no row of length $Y_m + |\eta^m_2|$. Let $g_1$ denote the bottommost row of length $Y_m + |\eta^m_1|$, and let $g_2$ denote the stretch of $\lambda_{n-i-1}$ contained in $g_1$. Since $g_2$ contains $\mathrm{cb}(\lambda_{n-i-1})_1^{b_m}$ contributing boxes and $|g_2| = Y_m + |\eta^m_1| - |U_{b_m}|$, we obtain the formula for $r_{n-i-1}^{b_m,1}$.
\item In the last case, Let $t_1$ denote the bottommost row of length $Y_m + |\eta^m_2|$, and let $t_2$ denote the stretch of $\lambda_{n-i-1}$ contained in $t_1$. Since $t_2$ contains $\mathrm{cb}(\lambda_{n-i-1})_2^{b_m}$ contributing boxes and $|t_2| = \mathrm{cb}(\lambda_{n-i-1})_2^{b_m}$, we obtain the formula for $r_{n-i-1}^{b_m,2}$.
\end{enumerate}
The other riggings are easily expressed in terms of the rigging of the bottommost row beneath $s^{b_m}$.
\item The first row $\lambda_{n-i-1}^1$ of $\lambda_{n-i-1}$ is determined as follows.
\begin{enumerate}
\item If $|\eta^k_1| < |\tilde{\lambda}_{n-i}^1|$ and $|\lambda_{n-i}^1| > |\lambda_{n-i}^2|$, then no more contributing boxes can be added after $\eta^k_1$ in the first row, so $|\lambda_{n-i-1}^1| = |\lambda_{n-i}^2| + |\eta^k_1|$, with rigging $r_{n-i-1}^{b_k,1}$.
\item If $|\eta^k_1| = |\tilde{\lambda}_{n-i}^1|$ or $|\lambda_{n-i}^1| = |\lambda_{n-i}^2|$, then any number $\mathrm{cb}'(\lambda_{n-i-1})^{b_k}$ of contributing boxes can be added after $\eta^k_1$ in the first row without affecting the riggings of $\lambda_{n-i}$. 
\begin{enumerate}
\item Since stretch $\tilde{\lambda}_{n-i-1}^1$ contains $\mathrm{cb}'(\lambda_{n-i-1})^{b_k}$ contributing boxes and has length $\mathrm{cb}'(\lambda_{n-i-1})^{b_k} = |\lambda_{n-i-1}^1| - |\lambda_{n-i-1}^2|$, we obtain the formula for $r_{n-i-1}^{b_k,!}$.
\item Since $|\eta^k_1| = |\tilde{\lambda}_{n-i}^1|$ and stretch $\tilde{\lambda}_{n-i-1}^1$ contains $\mathrm{cb}(\lambda_{n-i-1})_1^{b_k} + \mathrm{cb}'(\lambda_{n-i-1})^{b_k}$ contributing boxes and has length $|\lambda_{n-i-1}^1| - |\lambda_{n-i-1}^2|$, we obtain the formula for $r_{n-i-1}^{b_k,!}$.
\end{enumerate}
\end{enumerate}
\end{enumerate}
\end{enumerate}
\end{proof}

\begin{exm}
\label{test2}
Let $$\lambda_n = {\small\young(\ \ \ )}$$ with rigging $-3 + \mathrm{acon}_n$, and fix $\mathrm{acon}_n = 2$. Then one possible choice of $\lambda_{n-1}$ is $$\lambda_{n-1} = {\small\young(\ \ \ \ \ ,\ \ )},$$ where the bottom row has rigging $-2 + \mathrm{acon}_{n-1}^1$ and the top row has rigging $-4 + \mathrm{acon}_{n-1}^2$. This is obtained by adding two contributing boxes to the second row and two contributing boxes to the end of the first row (whose first three boxes are noncontributing boxes).

Suppose we now fix $\mathrm{acon}_{n-1}^1 = 2$ and $\mathrm{acon}_{n-1}^2 = 4$. Then one possible choice for $\lambda_{n-2}$ is $$\lambda_{n-2} = {\small\young(\ \ \ \ \ \ ,\ \ \ ,\ \ )}$$ where the third row has rigging $-2 + \mathrm{acon}_{n-2}^1$, second row has rigging $-3 + \mathrm{acon}_{n-2}^1 + \mathrm{acon}_{n-2}^2$, first row has rigging $-5 +  \mathrm{acon}_{n-2}^1 + \mathrm{acon}_{n-2}^2 + \mathrm{acon}_{n-2}^3$. This choice for $\lambda_{n-2}$ is constructed by first starting with $\lambda_{n-2}^* = {\small\young(\ \ )}$ (with rigging $0$), then adding two noncontributing boxes to the second row, two contributing boxes to the third row, one noncontributing box to the first row, one contributing box to the second row, one noncontributing box to the first row, two contributing boxes to the first row, in that order.
\end{exm}

\subsection{Determining the Cascading Sequence of a Rigged Configuration}
\label{cs from rigged}
Based on Theorem \ref{growth alg thm}, we now give the algorithm for determining the cascading sequence of a rigged configuration. Assume $\Lambda = (\lambda_1,\lambda_2,\ldots,\lambda_n)$ is a ${\cal B}(\infty)$ rigged configuration of $A$-type.
\begin{thm}
\label{get seq}
The following algorithm constructs the cascading sequence $\alpha$ corresponding to $\Lambda$:
\begin{enumerate}
\item Start with the empty string $\alpha^0$. Add $|\lambda_n|$ copies of lower subintervals $(n)$ to $\alpha^0$, obtaining $\alpha^1$, which accounts for $\lambda_n$.
\item In general, suppose that we have constructed the cascading sequence $\alpha^i$ which accounts for $\lambda_n, \lambda_{n-1}, \ldots, \lambda_{n-i}$. We want to construct $\alpha^{i+1}$ that accounts for $\lambda_n,$ $\lambda_{n-1},$ $\ldots,$ $\lambda_{n-i},$ $\lambda_{n-i-1}$. \\
Label the stretches of $\lambda_{n-i}$ by $\tilde{\lambda}_{n-i}^{b_1},$ $\tilde{\lambda}_{n-i}^{b_2},$ $\ldots,$ $\tilde{\lambda}_{n-i}^{b_k}$, with corresponding rows $\lambda_{n-i}^{b_1},$ $\lambda_{n-i}^{b_2},$ $\ldots,$ $\lambda_{n-i}^{b_k}$ respectively. Write the rigging of row $\lambda_{n-i}^{b_j}$ as $$r_{n-i}^{b_j} = \sum_{m=1}^{j}{-\mathrm{cb}(\lambda_{n-i})^{b_m} + \mathrm{acon}_{n-i}^{b_m}}.$$\\
Label the stretches of $\overline{\lambda_{n-i}}$ by $s^{b_1}, s^{b_2}, \ldots, s^{b_{k}}$. Let $w^{b_1}, w^{b_2}, \ldots, w^{b_{k}}$ denote the stretches of the copy of $\overline{\lambda_{n-i}}$ sitting inside $\lambda_{n-i-1}$.
\begin{enumerate}
\item For $m = 1, 2, \ldots, k$, let $l_m$ denote the number of boxes in the second row beneath $w^{b_m}$ and let $l_m \leq u_m \leq |w^{b_m}|$ denote the number of boxes in the first row beneath $w^{b_m}$. For $m$ ranging through $1, 2, \ldots, k$ in that order, first apply Procedure \ref{addnoncon} to add $l_m + u_m-\mathrm{acon}_{n-i}^{b_m}$ noncontributing boxes beneath $s^{b_m}$, then apply Procedure \ref{addcon} to add $l_m$ contributing boxes beneath these added noncontributing boxes, and finally apply Procedure \ref{addcon} to add $\mathrm{acon}_{n-i}^{b_m}-l_m$ contributing boxes to the first row beneath $s^{b_m}$, updating the cascading sequence (starting from $\alpha^i$) with each application of each procedure.
\item Suppose we have added all the boxes required beneath the stretches of $\overline{\lambda_{n-i}}$. Let $g$ denote the resulting first row. Apply Procedure \ref{addcon} to add $|\lambda_{n-i-1}^1| - |g|$ contributing boxes to $g$, updating the cascading sequence. This completes the construction of $\lambda_{n-i-1}$, and the resulting cascading sequence is the desired $\alpha^{i+1}$.
\end{enumerate}
\end{enumerate}
\end{thm}
\begin{proof}
By assumption, $\Lambda$ is a legitimate rigged configuration. This algorithm works by comparing $\Lambda$ with the rigged configuration corresponding to the cascading sequence constructed so far, seeing what boxes need to be added to construct the next partition of $\Lambda$, and then applying Procedure \ref{addnoncon} and Procedure \ref{addcon} to add the boxes required. The full proof is similar to that of Theorem \ref{growth alg thm}, and is a matter of bookkeeping.
\end{proof}

Now let us look at some examples of how to obtain the cascading sequence given a rigged configuration using the algorithm described above.
\begin{exm}
Consider the following $A_{10}$ rigged configuration $R = (\nu_1,\nu_2,\ldots,\nu_{10})$ (in top-bottom order) where $\nu_{i}$ is the $i$th rigged partition whose $j$th row has rigging $\mathrm{rig}_i^j$: \\
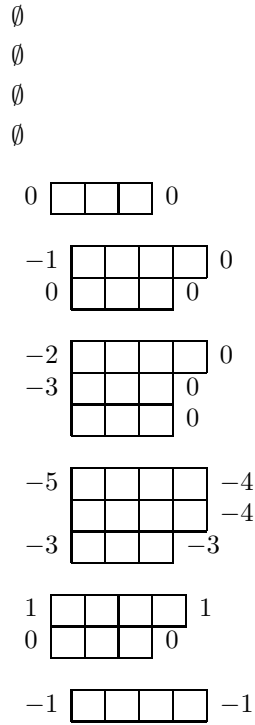
\begin{align*}
&\emptyset\\
\vspace{5 mm}
&\emptyset\\
\vspace{5 mm}
&\emptyset\\
\vspace{5 mm}
&\emptyset\\
\vspace{5 mm}
&{
\begin{array}[t]{r|c|c|c|l}
 \cline{2-4} 0 &\phantom{|}&\phantom{|}&\phantom{|}& 0 \\
 \cline{2-4} 
\end{array}
} \\
\vspace{5 mm}
&{
\begin{array}[t]{r|c|c|c|c|l}
 \cline{2-5} -1 &\phantom{|}&\phantom{|}&\phantom{|}&\phantom{|}& 0 \\
 \cline{2-5} 0 &\phantom{|}&\phantom{|}&\phantom{|}& \multicolumn{2 }{l}{ 0 } \\
 \cline{2-4} 
\end{array}
} \\
\vspace{5 mm}
&{
\begin{array}[t]{r|c|c|c|c|l}
 \cline{2-5} -2 &\phantom{|}&\phantom{|}&\phantom{|}&\phantom{|}& 0 \\
 \cline{2-5} -3 &\phantom{|}&\phantom{|}&\phantom{|}& \multicolumn{2 }{l}{ 0 } \\
 \cline{2-4}  &\phantom{|}&\phantom{|}&\phantom{|}& \multicolumn{2 }{l}{ 0 } \\
 \cline{2-4} 
\end{array}
} \\
\vspace{5 mm}
&{
\begin{array}[t]{r|c|c|c|c|l}
 \cline{2-5} -5 &\phantom{|}&\phantom{|}&\phantom{|}&\phantom{|}& -4 \\
 \cline{2-5}  &\phantom{|}&\phantom{|}&\phantom{|}&\phantom{|}& -4 \\
 \cline{2-5} -3 &\phantom{|}&\phantom{|}&\phantom{|}& \multicolumn{2 }{l}{ -3 } \\
 \cline{2-4} 
\end{array}
} \\
\vspace{5 mm}
&{
\begin{array}[t]{r|c|c|c|c|l}
 \cline{2-5} 1 &\phantom{|}&\phantom{|}&\phantom{|}&\phantom{|}& 1 \\
 \cline{2-5} 0 &\phantom{|}&\phantom{|}&\phantom{|}& \multicolumn{2 }{l}{ 0 } \\
 \cline{2-4} 
\end{array}
} \\
\vspace{5 mm}
&{
\begin{array}[t]{r|c|c|c|c|l}
 \cline{2-5} -1 &\phantom{|}&\phantom{|}&\phantom{|}&\phantom{|}& -1 \\
 \cline{2-5} 
\end{array}
}
\end{align*}

From the viewpoint of its cascading sequence, $R$ is constructed (by the growth algorithm) in the following process (where newly added letters or lower subintervals at each stage are marked with a prime ($'$)): 
\begin{align*}
(10')(10')(10')(10') &\rightarrow^{\textcircled{1}} (9',10)(9',10)(9',10)(9',10) \\ 
&\rightarrow^{\textcircled{2}} (9,10)(9,10)(9,10)(9,10)(8,9)'(8,9)'(8,9)' \\
&\rightarrow^{\textcircled{3}} (8',9,10)(8',9,10)(8',9,10)(9,10)(7',8,9)(7',8,9)(7',8,9) \\ 
&\rightarrow^{\textcircled{4}} (8,9,10)(8,9,10)(8,9,10)(9,10)(7,8,9)(7,8,9)(7,8,9) \\
&\phantom{{}=1}(6,7,8)'(6,7,8)'(6,7,8)' \\ 
&\rightarrow^{\textcircled{5}} (8,9,10)(8,9,10)(8,9,10)(8',9,10)(7,8,9)(7,8,9)(7,8,9) \\
&\phantom{{}=1}(6,7,8)(6,7,8)(6,7,8) \\ 
&\rightarrow^{\textcircled{6}} (8,9,10)(8,9,10)(8,9,10)(8,9,10)(7,8,9)(7,8,9)(7,8,9) \\
&\phantom{{}=1}(6,7,8)(6,7,8)(6,7,8)(7,8)' \\ 
&\rightarrow^{\textcircled{7}} (7',8,9,10)(7',8,9,10)(7',8,9,10)(8,9,10)(6',7,8,9) \\
&\phantom{{}=1}(6',7,8,9)(6',7,8,9)(5',6,7,8)(5',6,7,8)(5',6,7,8)(7,8) \\
&\rightarrow^{\textcircled{8}} (7,8,9,10)(7,8,9,10)(7,8,9,10)(8,9,10)(6,7,8,9) \\
&\phantom{{}=1}(6,7,8,9)(6,7,8,9)(5,6,7,8)(5,6,7,8)(5,6,7,8)(6',7,8)
\end{align*}
\textbf{Explanation of the above process}: We started out by adding four $10$-boxes, which completes Partition 10. Since $\mathrm{rig}_{10}^1 = -1 = -4 + 3$, we first added four noncontributing $9$-boxes in $\textcircled{1}$, and then added three contributing $9$-boxes in $\textcircled{2}$ beneath these noncontributing boxes, which completes Partition 9 and adds three noncontributing $8$-boxes. Since $\mathrm{rig}_9^2 = 0 = -3 + 3$ and $\mathrm{rig}_9^1 = 1 = -3 + 4$, we first added three noncontributing $8$-boxes beneath the first row in  $\textcircled{3}$ (along with three noncontributing $7$-boxes), and then added three contributing $8$-boxes beneath the second row in $\textcircled{4}$ (along with three noncontributing $7$-boxes and three noncontributing $6$-boxes), and then added one noncontributing $8$-box to the first row in $\textcircled{5}$, and then added one contributing $8$-box beneath the first row in $\textcircled{6}$ (along with one noncontributing $7$-box to the first row). This completes Partition 8. Now, Partitions 5-7 all have zero riggings, while the remaining partitions are empty. To complete Partition 7, we added three noncontributing $7$-boxes beneath the second row in $\textcircled{7}$ (along with three noncontributing $6$-boxes to the second row and three noncontributing $5$-boxes to the first row). Finally, we added one noncontributing $6$-box to the first row in $\textcircled{8}$ to complete Partition 6. This gives us the desired rigged configuration.
\end{exm}

\begin{exm}
Consider the following $A_{10}$ rigged configuration $S = (\nu_1,\nu_2,\ldots,\nu_{10})$ (in top-bottom order) where $\nu_{i}$ is the $i$th rigged partition whose $j$th row has rigging $\mathrm{rig}_i^j$: \\
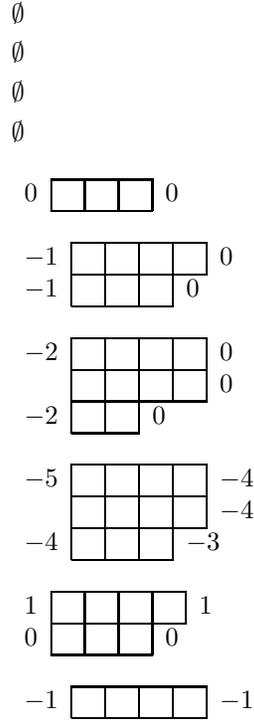
\begin{align*}
&\emptyset\\
\vspace{5 mm}
&\emptyset\\
\vspace{5 mm}
&\emptyset\\
\vspace{5 mm}
&\emptyset\\
\vspace{5 mm}
&{
\begin{array}[t]{r|c|c|c|l}
 \cline{2-4} 0 &\phantom{|}&\phantom{|}&\phantom{|}& 0 \\
 \cline{2-4} 
\end{array}
} \\
\vspace{5 mm}
&{
\begin{array}[t]{r|c|c|c|c|l}
 \cline{2-5} -1 &\phantom{|}&\phantom{|}&\phantom{|}&\phantom{|}& 0 \\
 \cline{2-5} -1 &\phantom{|}&\phantom{|}&\phantom{|}& \multicolumn{2 }{l}{ 0 } \\
 \cline{2-4} 
\end{array}
} \\
\vspace{5 mm}
&{
\begin{array}[t]{r|c|c|c|c|l}
 \cline{2-5} -2 &\phantom{|}&\phantom{|}&\phantom{|}&\phantom{|}& 0 \\
 \cline{2-5}  &\phantom{|}&\phantom{|}&\phantom{|}&\phantom{|}& 0 \\
 \cline{2-5} -2 &\phantom{|}&\phantom{|}& \multicolumn{3 }{l}{ 0 } \\
 \cline{2-3} 
\end{array}
} \\
\vspace{5 mm}
&{
\begin{array}[t]{r|c|c|c|c|l}
 \cline{2-5} -5 &\phantom{|}&\phantom{|}&\phantom{|}&\phantom{|}& -4 \\
 \cline{2-5}  &\phantom{|}&\phantom{|}&\phantom{|}&\phantom{|}& -4 \\
 \cline{2-5} -4 &\phantom{|}&\phantom{|}&\phantom{|}& \multicolumn{2 }{l}{ -3 } \\
 \cline{2-4} 
\end{array}
} \\
\vspace{5 mm}
&{
\begin{array}[t]{r|c|c|c|c|l}
 \cline{2-5} 1 &\phantom{|}&\phantom{|}&\phantom{|}&\phantom{|}& 1 \\
 \cline{2-5} 0 &\phantom{|}&\phantom{|}&\phantom{|}& \multicolumn{2 }{l}{ 0 } \\
 \cline{2-4} 
\end{array}
} \\
\vspace{5 mm}
&{
\begin{array}[t]{r|c|c|c|c|l}
 \cline{2-5} -1 &\phantom{|}&\phantom{|}&\phantom{|}&\phantom{|}& -1 \\
 \cline{2-5} 
\end{array}
}
\end{align*}

From the viewpoint of cascading sequences, $S$ is constructed in the following process:
\begin{align*}
(10')(10')(10')(10') &\rightarrow^{\textcircled{1}} (9',10)(9',10)(9',10)(9',10) \\ 
&\rightarrow^{\textcircled{2}} (9,10)(9,10)(9,10)(9,10)(8,9)'(8,9)'(8,9)' \\ 
&\rightarrow^{\textcircled{3}} (8',9,10)(8',9,10)(8',9,10)(9,10)(7',8,9)(7',8,9)(7',8,9) \\ 
&\rightarrow^{\textcircled{4}} (8,9,10)(8,9,10)(8,9,10)(9,10)(7,8,9)(7,8,9)(7,8,9) \\
&\phantom{{}=1}(6,7,8)'(6,7,8)'(6,7,8)' \\ 
&\rightarrow^{\textcircled{5}} (8,9,10)(8,9,10)(8,9,10)(8',9,10)(7,8,9)(7,8,9)(7,8,9) \\
&\phantom{{}=1}(6,7,8)(6,7,8)(6,7,8) \\ 
&\rightarrow^{\textcircled{6}} (8,9,10)(8,9,10)(8,9,10)(8,9,10)(7,8,9)(7,8,9)(7,8,9) \\
&\phantom{{}=1}(6,7,8)(6,7,8)(6,7,8)(7,8)' \\ 
&\rightarrow^{\textcircled{7}} (7',8,9,10)(7',8,9,10)(8,9,10)(8,9,10)(6',7,8,9)(6',7,8,9) \\
&\phantom{{}=1}(7,8,9)(5',6,7,8)(5',6,7,8)(6,7,8)(7,8) \\ 
&\rightarrow^{\textcircled{8}} (7,8,9,10)(7,8,9,10)(7',8,9,10)(8,9,10)(6,7,8,9)(6,7,8,9) \\
&\phantom{{}=1}(7,8,9)(5,6,7,8)(5,6,7,8)(6,7,8)(6',7,8) \\ 
&\rightarrow^{\textcircled{9}} (6',7,8,9,10)(7,8,9,10)(7,8,9,10)(8,9,10)(6,7,8,9) \\
&\phantom{{}=1}(6,7,8,9)(7,8,9)(5,6,7,8)(5,6,7,8)(5',6,7,8)(6,7,8)
\end{align*}
Explanation of the above process: \\
We started out by adding four $10$-boxes, which completes Partition 10. Since $\mathrm{rig}_{10}^1 = -1 = -4 + 3$, we first added four noncontributing $9$-boxes in $\textcircled{1}$, and then added three contributing $9$-boxes in $\textcircled{2}$ beneath these noncontributing boxes, which completes Partition 9 and adds three noncontributing $8$-boxes. Since $\mathrm{rig}_9^2 = 0 = -3 + 3$ and $\mathrm{rig}_9^1 = 1 = -3 + 4$, we first added three noncontributing $8$-boxes beneath the first row in  $\textcircled{3}$ (along with three noncontributing $7$-boxes), and then added three contributing $8$-boxes beneath the second row in $\textcircled{4}$ (along with three noncontributing $7$-boxes and three noncontributing $6$-boxes), and then added one noncontributing $8$-box to the first row in $\textcircled{5}$, and then added one contributing $8$-box beneath the first row in $\textcircled{6}$ (along with one noncontributing $7$-box to the first row). This completes Partition 8. Since $\mathrm{rig}_8^3 = -3 + 0$ and $\mathrm{rig}_8^1 = \mathrm{rig}_8^2 = -4 +0$, there are no contributing $7$-boxes to add. In $\textcircled{7}$, we added two noncontributing $7$-boxes to the third row. In $\textcircled{8}$, we added a noncontributing $7$-box to the second row. In $\textcircled{9}$, we added a noncontributing $6$-box to the second row. This completes Partition 6, and yields the desired rigged configuration.
\end{exm}

\section{Further Discussions}
One can try to characterize ${\cal B}(\infty)$ rigged configurations in the types $B, C, D, G$, by modifying or extending the methods used in this paper. One can also try to find a non-recursive characterization of ${\cal B}(\infty)$ rigged configurations, which describes the $i$th rigged partition without reference to the $(i+1)$st partition.

    %
    %

%
%

\end{document}